\documentclass[12pt,a4paper]{amsart}
\usepackage{amscd}
\usepackage{amssymb}
\usepackage{amsmath}
\usepackage{amsfonts}
\usepackage[all]{xypic}
\usepackage{hyperref}

\usepackage[dvips,pdftitle={Divisors, Layout em LaTeX, Maio 2012},
pdfauthor={ALC, CELC}]{}

\numberwithin{equation}{section}

\theoremstyle{plain}
    \newtheorem{theorem}{Theorem}[section]
    \newtheorem{corollary}[theorem]{Corollary}
    \newtheorem{lemma}[theorem]{Lemma}
    \newtheorem{proposition}[theorem]{Proposition}
    \newtheorem{prop}[theorem]{Properties}
    \newtheorem*{theorem*}{Theorem}
    \newtheorem*{proposition*}{Proposition}
    \newtheorem*{corollary*}{Corollary}

\theoremstyle{definition}

    \newtheorem*{obs*}{Notation}
    
    \newtheorem{remark}[theorem]{Remark}
    \newtheorem*{remarks*}{Remarks}
    \newtheorem*{remark*}{Remark}

    \newtheorem{clas*}{Exercises}
    
    \newtheorem*{claimno*}{Examples}
    \newtheorem*{claim*}{Example}

    \newcommand{\reft}[1]{Theo\-rem~\ref{#1}}
    \newcommand{\refp}[1]{Pro\-po\-si\-tion~\ref{#1}}
    \newcommand{\refl}[1]{Lem\-ma~\ref{#1}}
    \newcommand{\refc}[1]{Co\-rol\-lary~\ref{#1}}

    \newcommand{\refeq}[1]{Eq.~(\ref{#1})}

\newcommand{\K}{\mathbb K}

\newcommand{\E}{\mathbb E}

\newcommand{\bZ}{\mathbf{Z}}

\newcommand\alp{\alpha}

\newcommand\tet{\theta}

\newcommand\var{\varphi}

\newcommand{\nor}{Noe\-the\-rian ring}

\newcommand{\nol}{Noe\-the\-rian lo\-cal ring}

\newcommand{\rmo}{$R$-mo\-du\-le}
\newcommand{\ff}{faith\-fully flat}
\newcommand{\fgr}{fi\-ni\-te\-ly ge\-ne\-ra\-ted $R$-mo\-du\-le}
\newcommand{\wrt}{with respect to\ }

\newcommand{\fg}{fi\-ni\-te\-ly ge\-ne\-ra\-ted}

\newcommand{\gbiu}{ge\-ne\-ric Bour\-baki ideal of $E$ with res\-pect
to $U$}

\newcommand{\ggbie}{good ge\-ne\-ric Bour\-baki ideal of $E$}
\newcommand{\gbi}{ge\-ne\-ric Bour\-baki ideal}
\newcommand{\ggbi}{good ge\-ne\-ric Bour\-baki ideal}

\newcommand{\gbie}{ge\-ne\-ric Bour\-baki ideal of $E$}

\newcommand{\eq}{equi\-mul\-ti\-ple}
\newcommand{\ci}{com\-ple\-te in\-ter\-sec\-tion}

\newcommand{\li}{li\-near\-ly in\-de\-pen\-dent}
\newcommand{\tf}{tor\-sion\-free }
\newcommand{\fgtfrmo}{fi\-ni\-te\-ly ge\-ne\-ra\-ted tor\-sion\-free
$R$-mo\-du\-le }
\newcommand{\fgtfrmos}{fi\-ni\-te\-ly ge\-ne\-ra\-ted tor\-sion\-free
$R$-mo\-du\-les }
\newcommand{\fgrmo}{fi\-ni\-te\-ly ge\-ne\-ra\-ted $R$-mo\-du\-le}
\newcommand{\Wlog}{Wi\-thout loss of ge\-ne\-ra\-li\-ty }

\newcommand{\seq}[2]{({#1}_1,\dotsc,{#1}_{#2})}

\newcommand{\seg}[2]{\langle{#1}_1,\dotsc,{#1}_{#2}\rangle}
\newcommand{\sek}[2]{{#1}_1,\dotsc,{#1}_{#2}}

\newcommand{\pr}{^{\prime}}
\newcommand{\prd}{^{\prime\prime}}

\newcommand{\mat}{\begin{bmatrix}}
\newcommand{\emat}{\end{bmatrix}}
\newcommand{\edet}{\end{matrix} \right|}

\newcommand{\tor}{\operatorname{tor}}

\newcommand{\deter}{\operatorname{det}}
\newcommand{\bca}{\begin{cases}}
\newcommand{\eca}{\end{cases}}

\renewcommand{\det}{\left| \begin{matrix}}

\newcommand{\col}{\colon}

\newcommand{\apl}[3]{#1 \col #2 \ra #3}

\newcommand{\x}{\times}
\newcommand{\ox}{\otimes}
\newcommand{\oxr}{\otimes_{R}}
\newcommand{\+}{\oplus}
\newcommand{\bop}{\bigoplus}

\newcommand{\sube}{\subseteq}
\newcommand{\subn}{\subsetneq}
\newcommand{\nsub}{\nsubseteq}

\newcommand{\sub}{\subset}

\newcommand{\sse}{if and only if\ }

\newcommand{\ra}{\rightarrow}

\newcommand{\lra}{\longrightarrow}
\newcommand{\Lra}{\Longrightarrow}

\newcommand{\hra}{\hookrightarrow}
\newcommand{\tra}{\twoheadrightarrow}

\newcommand{\wdg}{\wedge}
\newcommand{\bwd}{\bigwedge}
\newcommand{\tbwd}[1]{\textstyle{\,\bigwedge^{#1}}\,}

\newcommand{\cP}{{\mathcal P}}

\newcommand{\cE}{{\mathcal E}}

\newcommand{\cS}{{\mathcal S}}

\newcommand{\CR}{{\mathcal R}}

\newcommand{\cF}{{\mathcal F}}

\newcommand{\cO}{{\mathcal O}}

\newcommand{\fm}{{\mathfrak{m}}}

\newcommand{\frp}{{\mathfrak{p}}}

\newcommand{\frk}{k}

\DeclareMathOperator{\im}{im}
\DeclareMathOperator{\coker}{coker}

\DeclareMathOperator{\Id}{id}
\DeclareMathOperator{\rank}{rank}

\DeclareMathOperator{\spec}{Spec}

\DeclareMathOperator{\Ht}{ht}

\DeclareMathOperator{\supp}{Supp}

\DeclareMathOperator{\ann}{ann}

\DeclareMathOperator{\Quot}{Quot}

\DeclareMathOperator{\proj}{proj}
\DeclareMathOperator{\depth}{depth}
\DeclareMathOperator{\grade}{grade}

\DeclareMathOperator{\ad}{ad}

\DeclareMathOperator{\Proj}{Proj}

\DeclareMathOperator{\Bl}{Bl}

\newcommand{\Rm}{(R,\fm)}
\newcommand{\Rmk}{(R,\fm, \frk)}

\DeclareMathOperator{\Q}{Q}

\newcommand{\pjd}{\proj \dim \,}
\newcommand{\dpt}{\depth \,}
\newcommand{\grd}{\grade \,}

\newcommand{\lE}{\ell(E)}
\newcommand{\RE}{\CR(E)}
\newcommand{\FE}{\cF(E)}

\newcommand{\lI}{\ell(I)}

\newcommand{\ol}{\overline}

\allowdisplaybreaks

\begin{document}

\title[Divisors of a module and blow up]
{Divisors of a module and blow up}
\author{Ana L. Branco Correia}
\address{Centro de Estruturas Lineares e Combinat\'orias
\\ Universidade de Lisboa
\\ Av. Prof. Gama Pinto 2  \\ 1649-003 Lisboa
\\ Portugal}
\email{alcorreia@cii.fc.ul.pt}
\thanks{The research of the first author  was made within the activities of the Centro de Estruturas Lineares e Combinat—rias and was partially supported by the Funda\c c\~ao para a Ci\^encia e Tecnologia (Strategic Project PEst-OE/MAT/UI1431/2011).}
\address{ Universidade Aberta \\
Departamento de Ci\^encias e Tecnologia \\
Pal\'acio Ceia \\ Rua da Escola Polit\'ecnica, 147 \\
1269-001 Lisboa \\ Portugal}
\email{matalrbc@uab.pt}
\thanks{}
\thanks{}
\author{Santiago Zarzuela}
\address{Departament d'\`Algebra i Geometria \\ Universitat de
Barcelona \\ Gran Via 585 \\E-08007 Barcelona \\ Spain}
\email{zarzuela@mat.ub.es}
\thanks{The second author was supported by MTM2010-20279-C02-01}
\thanks{}
\date{\today}
\dedicatory{}
\commby{}


\begin{abstract}
In this paper we work with several divisors of a module $E \subseteq G \simeq R^{e}$ having rank $e$, such as the classical Fitting ideals of $E$ and of $G/E$, and the more recently introduced (generic) Bourbaki ideals $I(E)$ (A. Simis, B. Ulrich, and W. Vasconcelos \cite{suv}) or ideal norms  $[[E]]_R$ (O. Villamayor \cite{ov}). We found several relations and equalities among them which allow to describe in some cases universal properties with respect to $E$ of their blow ups. As a byproduct we are also able to obtain lower bounds for the analytic spread $\ell(\bwd^{e}E)$, related with the algebraic local version of Zak's inequality as explained in A. Simis, K. Smith and B. Ulrich \cite{ssu}.
\end{abstract}


\maketitle

\section{Introduction}

\noindent
Paraphrasing W. Vasconcelos \cite{wv2},
the {\it divisors} of a module $E$ over a ring $R$ are the ideals of $R$ that carry important information about the structure and properties of $E$. Let $E$ be a \fgtfrmo having rank $e$ with a fixed embedding $f \colon E \sube G \simeq R^e$. Then the Fitting ideals of $E$ and of $G/E$ are typical examples. They determine the free locus, and they give rise in particular to the notions of principal class, equimultiple, complete intersection, etc.
Another divisor ideal that play a significant role in the study of the integral closure of E is $\deter_f(E)$, introduced by Rees in \cite{Rees1}. By definition $\deter_f(E)$ is the ideal defined by the image of the mapping $\bwd^e f$ in $\bwd^e R^e \simeq R$
and up to isomorphism $\deter_f(E)= \bwd^e E/\tau_{R}(\bwd^e R^e)$. In \cite{wv2} the ideal $\deter_f(E)$ is also denoted by $\deter_0(E)$, since it does not depend on $f$.
It is useful for instance to lead with reductions of modules. In fact, if $U \sube E$ are modules of the same rank over a domain, then $U$ is a reduction of $E$ if and only if $\deter_0(U)$ is a reduction of $\deter_0(E)$. See also \cite{hs} or \cite{wv2} for further details.

\smallbreak

On the other hand, O. Villamayor attached in \cite{ov} a new class of ideals to $E$, called the norm of $E$ and denoted by $[[E]]_R$, and showed that the blow up at any representant of this class has a universal flattening property.
That is, if $X \overset{\pi}{\lra}\spec(R)$ is the blow up of $R$ with center $[[E]]_R$, then
$\pi^{*}(E)/\tor(\pi^{*}(E))$ is a locally free sheaf of $\cO_{X}$-modules of rank $e$, and $\pi$ is universal for this property.
He named this blow up as the blow up of $R$ at $E$. In the analytical setting the construction of a homomorphism with such flattening property was established by H. Rossi in \cite{Rossi}, who asked for the determination of this construction by some universal property. In fact, under the name of Nash transformation of $R$ at $E$ the existence in general of this universal construction was established by A. Oneto and E. Zatini in \cite{oz}, obtaining as a particular case of this construction the usual Nash blowing-up. Similarly, the higher Nash blowups defined by T. Yasuda in \cite{ya} may be seen as particular cases of these general Nash transformations. And also the $e$-th $F$-blow ups, $e$ a non-negative integer, of a variety $X$ of positive characteristic as defined by Yasuda himself in \cite{ya2}.

\smallbreak

The main novelty in Villamayor's construction is that he showed that there exists a module $E_{1}$ of projective dimension at most $1$ having rank $e$ whose Fitting ideal $F_{e}(E_{1})$ is a representative of $[[E]]_R$, as a fractional ideal. This Fitting ideal $F_{e}(E_{1})$ is in fact a sub-determinantal ideal of any matrix representing $E$, which gives an effective method to compute the blow up of $R$ at $E$. For instance, N. Hara, T. Sawada and T. Yasuda have recently used this approach in \cite{hsy} in order to compute explicitly $F$-blowups for some surface singularities by using Macaulay2 \cite{M2}.

\smallbreak

In section \ref{div}, we make clear the analogies of these ideals by showing that $\deter_{0}(E) = F_{0}(G/E)$ is also a representant of $[[E]]_{R}$. To do this we explore the relationship between Fitting ideals and determinant ideals with exterior algebras. Moreover, by using that $\deter_{0}(E)\deter_{0}(E)^{-1}$ defines the non-free locus of $E$ and that $\deter_{0}(E) \cdot S$ is invertible if and only if $E\oxr S /\tau_{R}(E\oxr S)$ is free for any birational extension $S$ of $R$, we may also show that the blow of $R$
at $\deter_{0}(E)$ has the universal flattening property, providing another proof for the existence of this universal object, somehow in the spirit of \cite[section 2]{oz}.

\smallbreak

As a byproduct of the above study we get some results concerning the algebraic local version of the so-called Zak's inequality. Let $R$ be a \nor{} and $E$ a \fgtfrmo having a rank. We define the {\it Rees algebra} $\RE$ of $E$ to be the quotient $\cS(E)/\tau_{R}(\cS(E))$ of the symmetric algebra by its $R$-torsion submodule $\tau_{R}(\cS(E))$. $\RE$ inherits a natural graduation from the symmetric algebra and we denote by $E^{n}$ the $n$th graded piece, that is, $E^{n}:=\RE_{n}$. If $\Rmk$ is also a local ring the {\it fiber cone} of $\RE$ is the graded ring $\FE := \RE \oxr k = \bigoplus_{n\geq 0} E^{n}/\fm E^{n}$. The Krull dimension of $\cF(E)$ is called the {\it analytic spread} of $E$ and is denoted by $\lE$.
In \cite{su}, A. Simis and B. Ulrich discussed the problem $$ \text{``when does the inequality $\ell\big(\tbwd{e}E\big) \geq \Ht F_{e}(E)$ hold?''} $$ posed in \cite{ssu}. This problem is related to the named Zak's inequality for the dimension of the image of the Gauss map when the variety is not smooth in terms of the dimension of its singular locus. Using the relationships obtained in section \ref{div} for the considered divisor ideals,
we also give some affirmative answers for this question (cf. Proposition \ref{zak} and Corollary \ref{zak3}).

\smallbreak

For some modules $E$ it is possible to find a free submodule $F$
such that the quotient $E/F$ is isomorphic to an
ideal $I$. For example this happens whenever $R$ is a Noetherian normal
domain
and $E$ is a \tf \rmo{} (cf. \cite[Chap. 7, \S4, Th\'eor\`eme 6]{b1}).
If it exists, the ideal $I \simeq E/F$ is called a {\it Bourbaki ideal} of $E$.
More generally an exact sequence of $R$-modules
$$ 0 \ra F \ra E \ra I \ra 0 $$
with $F$ a free module and $I$ an $R$-ideal is called a {\it Bourbaki
sequence}.
In \cite{suv}, A. Simis, B. Ulrich, and W. Vasconcelos developed the notion of generic Bourbaki ideals defined over a suitable faithfully flat extension of $R$. This approach has the advantage that the Bourbaki ideal thus obtained is essentially unique. As a matter of fact, they proved that for some families of modules $E$ over a local ring $R$ and ha\-ving rank there is a suitable Nagata extension $R\prd$ of $R$, together with a free $R\prd$-module $F$ such that $E\prd/F= \ol{E}$ is isomorphic to an $R\prd$-ideal of rank $1$, where $E\prd=E\oxr R\prd$
(cf. \cite[Proposition~3.2]{suv}).
They wrote $I(E)$ to denote this ideal and called it a {\it generic
Bourbaki ideal of $E$}.
The induced epimorphism of $R\prd$-algebras
$\RE \oxr R\prd \simeq \CR(E\prd) \tra \CR(I(E))$
plays a major role throughout their work.
They used it to transfer properties about $\RE$ to $\CR(I(E))$
and vice-versa.
In fact, in the case where the kernel is generated by a regular
sequence on $\CR(E\prd)$, then $\CR(E\prd)$ is a {\it
deformation} of $\CR(I(E))$ and properties such as Cohen-Macaulayness
and normality can be transferred from $\CR(I(E))$ to $\CR(E\prd)$, and hence to
$\RE$ (cf. \cite[Theorem~3.5]{suv}).

\smallskip

In \cite{cz3} we examined their construction of $I(E)$ and proved other properties. In par\-ti\-cu\-lar, we related the depth of $\RE$ and of $\CR(I(E))$.
In section \ref{qm} we establish other relations between $E$
and a $I(E)$. Especially, we obtain a formula relating the minimal number of generators of $E$ and of $I(E)$.

\begin{theorem*}{\rm [cf. \reft{free2}]}
    Let $R$ be a \nol, let $E$ be a \fg{} \rmo{} having
    rank $e \geq 2$ (and $U$ a reduction of $E$).
    If $I(E)$ is a \gbie{} (with respect to $U$) then
    \begin{equation}\label{mue}
    \mu(E)= \mu(I(E))+e-1.
    \end{equation}
\end{theorem*}
   This formula may be seen as an extension of a similar one for the analytic spread proved in \cite{suv}. In particular, Eq. (\ref{mue}) implies that there exists a generic Bourbaki $I(E)$ which is perfect of grade $2$ if and only if $E$ is not a free $R$-module and also $\pjd E=1
   $.
\smallbreak

In section \ref{gbifi}, we relate generic Bourbaki ideals with Fitting ideals for a \fgrmo{} $E$ having rank $e \geq 2$.
Supposing that $I\simeq E''/F$ is a \gbie{} over a Nagata extension $R''$ of $R$, we extend a basis $\sek{x}{e-1}$ of $F$ to a generating set $\sek{x}{n}$ of $E''=E\oxr R''$ and consider a finite free presentation $R''^m \overset{\varphi}{\ra} R''^n \overset{\phi}{\ra} E'' \ra 0$ with respect to this set of generators, that is $\mat x_{1} & \cdots & x_{n}\emat \varphi=0$. For this presentation there exists an $(n-e+1) \x (n-e)$ submatrix $\psi$ of $\varphi$ satisfying $\grd I_{n-e}(\psi) \geq 1$, and in the case where $\grd F_{e}(E) \geq 2$ then $E''/F \simeq I_{n-e}(\psi)$ (see Theorem \ref{fi}). As a consequence, we deduce that any generic Bourbaki ideal is always isomorphic to a Fitting ideal. It is important to note this Fitting ideal is also a sub-determinantal ideal of a specific presentation of $E''$.

\smallbreak

The interaction between their Rees algebras and the relations between numerical invariants of $E$ and $I(E)$ such as reduction number, analytic spread, minimal number of generators, justify that a generic Bourbaki ideal $I(E)$ may also be viewed as a divisor of $E$. In the final section \ref{div2} we return to the questions about the relations among the different divisor ideals of $E$ including now the generic Bourbaki ideal. The realization obtained by Villamayor of the norm ideal as a Fitting ideal is the key for that inclusion. Observing first that the norm ideal behaves well under the extension of scalars by flat homomorphisms, we have that
$$ I(E) \simeq I_{n-e}(\psi) \subseteq I_{n-e}(\rho) \subseteq F_{e}(E''),$$
with $\rho$ the $n\times (n-e)$ submatrix of $\varphi$ containing $\phi$ and  $I_{n-e}(\rho) \simeq [[E'']]_{R''} \simeq \deter_{0}(E)\cdot R''$. If moreover $\pjd E=1$ then $I_{n-e}(\rho) = F_{e}(E'')$. In particular, since if $E$ is a contracted module over a $2$-dimensional regular local ring it may be seen that $F_{e}(E'')$ itself is a generic Bourbaki ideal \cite{hu}, we get that there exists a generic Bourbaki ideal $I(E)$ of $E$ which is a representative of $[[E'']]_{R''}$ (cf. Corollary \ref{norma3}), and so it fulfills the corresponding universal flattening property with respect to $E''$.


\smallbreak

For any unexplained terminology we refer to the book of W. Bruns and J. Herzog \cite{bh}. Finally, we would like to thank T. Yasuda for telling us about the paper by Oneto-Zatini \cite{oz}.

\section{Divisors of a module - part 1}\label{div}

\subsection{Fitting ideals}

Let $E$ be a \fgrmo. Then the Fitting ideals of $E$
are typical examples of divisors of $E$.
By definition, $F_{i}(E):=I_{n-i}(\varphi)$ is the ideal generated by the $(n-i)\x (n-i)$ minors of $\varphi$ where
\begin{equation} \label{pE}
R^m \overset{\varphi}{\ra} R^n \overset{\phi}{\ra} E \ra 0
\end{equation}
is a finite presentation of $E$, $0 \leq i \leq n$.
This ideal is independent of the choices of the generators and of any finite free presentation (\ref{pE}) of $E$.
We have that $F_i(E) \subseteq F_{i+1}(E)$ and if, in addition, $E$ has rank $e$ (which means that $E \oxr \Quot(R)\simeq \Quot(R)^{e}$) then  $F_e(E)$ is the smallest Fitting ideal which is different from zero. Hence if $E$ has positive rank $e$ then \begin{equation*} \label{fid}
F_{0}(E)= \cdots = F_{e-1}(E)=(0) \subn F_e(E) \sube \cdots \sube F_{i}(E) \sube \cdots \sube R.
\end{equation*}

Now suppose that we have a given embedding $E \sube G \simeq R^e$ with $E$ of rank $e$. Then, $G/E$ has rank $0$ and the first non-zero Fitting ideals $F_e(E)$ and $F_0(G/E)$ are related by the inclusion $V(F_e(E)) \subseteq V(F_0(G/E)) = \supp G/E$. Recall that a torsionfree module $E \sube G \simeq R^e$ is an {\it ideal module} if $\grd G/E \geq 2$. (This condition implies that $E$ has rank $e$. Also that $E$ is not free if $E \subsetneq G$.) In this case,
\begin{equation}\label{FeF0}
V(F_{e}(E))=V(F_0(G/E))=\supp G/E.
\end{equation}
In particular, $\grd F_{e}(E)\geq 2$. See  \cite[section 3]{cz2} for further details.
In this section, we prove that $F_0(G/E) \subseteq F_e(E)$ (up to an isomorphism), and that the equality holds for ideal modules having projective dimension equal to one (see Proposition \ref{idmpj1}).

\subsection{Determinants}

Let $E$ be a \fgrmo{} and let $\apl{f}{E}{R^e}$ be an $R$-homomorphim.  The image of the mapping $\bwd^e f$ in $\bwd^e R^e \simeq R$ defines an $R$-ideal, that is $\im \bwd^e f = I \cdot \bwd^e R^e \sub \bwd^eR^e$,  where $I$ is an $R$-ideal, and we write (abusing notation slightly)
$$ \deter_f(E) := I =\im \tbwd{e} f \subseteq R.$$
This ideal is called the {\it determinant of the module $E$ with respect to $f$}.

The module $E$ does not need to be torsionfree. In fact, since $\tau_R(E) \subseteq \ker f$, we always have a commutative diagram
$$\xymatrix{ & E \ar@{->>}[dl]_{\pi} \ar[dr]^{f} & \\
E/\tau_R(E) \ar@{->}[rr]^{\ol{f}} & & R^e}$$
where $\pi$ is the canonical epimorphism and $\ol{f} \circ \pi = f$. We have the following pro\-per\-ties.

\begin{prop} \label{dfprop}
Let $R$ be a \nor. Let $E$ be a \fgrmo, $\apl{f}{E}{R^e}$ an $R$-homomorphism and $\apl{\varphi}{R}{S}$ a homomorphism of rings. Then
\begin{enumerate}
\item $\deter_f(E)=\deter_{\ol{f}}(E/\tau_R(E))$.
\item $\deter_{f\ox \Id}(E\oxr S)=\deter_f(E) \cdot S$.
\item $\deter_{f\ox \Id}(E\oxr S)=\deter_{\ol{f\ox \Id}}\big(E\oxr S/\tau_S(E \oxr S)\big).$
\end{enumerate}
\end{prop}

\begin{proof}
(a) $\apl{\bwd^{e}\pi}{\bwd^{e}E}{\bwd^{e}E/\tau_R(E)}$ is an epimorphism, hence $\im \bwd^e f = \im \bwd^e \ol{f}$ and (a) follows.
\smallbreak

(b) The homomorphism $f$ induces the commutative diagram
$$\xymatrix{ & E\oxr S \ar@{->>}[dl]_{\pi'} \ar[dr]^{f \ox \Id} & \\
E\oxr S/\tau_S(E\oxr S) \ar@{->}[rr]^{\ol{f\ox \Id}} & & R^e \oxr S \simeq S^e}$$
The exterior powers commute with base extensions, and so applying $\bwd^e$ we get commutative diagrams
\begin{equation*} \label{diag2}
\xymatrix{
\bwd^e E \ar[r]^{\bwd^e f} \ar[d] & \bwd^eR^e \simeq R \ar[d] \\
\bwd^e E \oxr S\ar[r]^{(\bwd^e f) \ox \Id} \ar[d]_{\simeq} &  \bwd^eR^e \oxr S \simeq S \ar[d]^{\simeq}\\
\bwd^e (E \oxr S) \ar[r]^{\bwd^e (f \ox \Id)}\ar@{->>}[d]_{\bwd^e \pi'}  &  \bwd^e(R^e \oxr S)\simeq S \\
  \bwd^e\big(E\oxr S)/\tau_R(E\oxr S)\big)\ar@{->}[ur]_{\bwd^e \ol{f \ox \Id}} &
}
\end{equation*}
Now
$$ \deter_{f\ox \Id}(E\oxr S) = \im \tbwd{e}(f \ox \Id) = (\im \tbwd{e} f)\cdot S = \deter_f(E)\cdot S.$$
\smallbreak
(c) follows by (a) and (b).
\end{proof}

We say that $f$ has rank if $\im f$ has rank. In this case, also $\ker f \subset E$ has rank. Assume that $E$ and $f$ have both rank $e$ and put $Q=\Quot(R)$ for the total ring of fractions of $R$. Then, $\rank(f\ox \Id)=e$ and $f \ox \Id : E \oxr Q \lra R^{e} \oxr Q$ is an isomorphism. Moreover, the homomorphism $\apl{\ol{f}}{E/\tau_R(E)}{R^e}$ is injective. This is a consequence of the following observation.

\begin{lemma} \label{lem1}
Let $R$ be a \nor. Let $E$ be a \fgtfrmo having rank $e$ and $\apl{f}{E}{R^e}$ an $R$-homomorphism with $\rank f=e$. Then $f$ is a monomorphism.
\end{lemma}

\begin{proof}
We have that $E\oxr Q \simeq f(E)\oxr Q \simeq Q^e$ and so $\ker f \oxr Q = 0$. This implies that $\ker f \subseteq \tau_{R}(E)=0$, proving that $f$ is a monomorphism.
\end{proof}

\subsection{Determinants and Fitting ideals}

In order to see determinants as Fitting ideals we fix some notation.
For a positive integer $n$, we set $[n] = \{ 1,\ldots, n \}.$
Let $H=\{ \sek{i}{h} \}$ be a subset of $[n]$ and suppose that
$i_{1} < i_{2} < \cdots < i_{h}$. We write
\begin{equation} \label{xH}
x_{H}= x_{i_{1}} \wdg x_{i_{2}} \wdg \cdots \wdg x_{i_{h}}.
\end{equation}
Let $A = (\alp_{ij})$ be an $n \x m$ matrix over a ring $R$.
For subsets $\{i_{1} < \cdots < i_{r}\} \sube [n]$ and $\{j_{1} < \cdots < j_{s}\} \sube [m]$ we write
$$ A[\sek{i}{r} \mid \sek{j}{s}] =
\begin{bmatrix}
    \alp_{i_{1},j_{1}} & \cdots & \alp_{i_{1},j_{s}}  \\
    \vdots &  & \vdots  \\
    \alp_{i_{r},j_{1}} & \cdots & \alp_{i_{r},j_{s}}
\end{bmatrix}. $$
Suppose that $E$ is generated by $\sek{x}{n}$. Suppose $m \leq n$. Hence $\{x_{H} \colon H \in \cP_{m}([n])\}$, where $\cP_{m}([n])$
denotes the set of all subsets of $[n]$ with $m$ elements, is the
corres\-pon\-ding generating set of $\bwd^m E$.
%
If $u_{j} = \sum_{i=1}^n \alp_{ij} x_{i}$ for $1\leq j \leq m$, then
\begin{equation}\label{pext1}
     u_{1} \wdg \cdots \wdg u_{m} = \sum_{H= \{ i_{1} < \cdots <
     i_{m} \} \in \cP_{m}([n])} \deter A[\sek{i}{m}|1,\ldots,m]
     x_{H}.
\end{equation}
Now, let $\seq{v}{e}$ be any basis of $R^{e}$. Hence $\bwd^{e}R^{e}= \langle v_{1} \wdg \cdots \wdg v_{e} \rangle \simeq R$. Let $\apl{f}{E}{R^{e}}$ be an $R$-homomorphism and suppose that $f(x_{i_{j}}) = \sum_{k=1}^{e}\alpha_{kj}v_{k}$ and put $A=(\alp_{ij})$.
Therefore
\begin{align*}
\im \tbwd{e}f & = \langle f(x_{i_{1}}) \wdg f(x_{i_{2}}) \wdg \cdots \wdg f(x_{i_{e}}) \colon i_{1} < i_{2} < \cdots < i_{e}\rangle \\
& = \langle \deter A[\sek{i}{e}|1,\ldots,e]\, v_{1}\wdg \cdots \wdg v_{e} \colon i_{1} < i_{2} < \cdots < i_{e}\rangle \\
& = \langle \deter A[\sek{i}{e}|1,\ldots,e] \colon i_{1} < i_{2} < \cdots < i_{e}\rangle \cdot \tbwd{e}R^{e}.
\end{align*}
It follows that $\deter_{f}(E)$ is the ideal generated by the $e \x e$ minors of $A=(\alp_{ij})$.
Moreover, since $E$ is generated by $n$ elements, we have a natural epimomorphism $R^{n} \tra E$. Hence
$$\xymatrix{R^{n} \ar@{->>}[r]\ar@{->}_{\psi}@/_{0.8pc}/[rr] & E \ar[r]^{\!\!\!\!\!\!\!\!\!f} & R^e=G \ar[r] & G/\im f}$$
is a finite presentation of $G/\im f$.
Therefore $F_0(G/\im f)=I_e(\psi)$ and we have
\begin{equation} \label{iguald}
\deter_f(E) = F_0(G/\im f).
\end{equation}

Now assume that $E$ has rank $e$ and let $Q = \Quot(R)$ be the total ring of fractions of $R$. Hence we have natural ho\-mo\-mor\-phisms
$$\tbwd{e} E \ra \big(\tbwd{e} E\big) \oxr Q  \simeq \tbwd{e} (E \oxr Q) \simeq \tbwd{e} Q^e \simeq Q.$$
Suppose also that $\apl{f}{E}{R^{e}}$ is an $R$-homomorphism such that $\im f$ has rank $e$.
(Note that if $E$ has rank $e$ one can always find such an $R$-homomorphism $f$. Also, that this is the case if $f: E \hra R^e$ is an embedding.)
Hence
we have a natural Q-isomorphism
$$ Q^e \simeq E \oxr Q \overset{f \ox \Id}{\lra} R^{e} \oxr Q \simeq Q^e $$
and natural homomorphisms
$$ \tbwd{e} E \overset{\bwd^{e} f}{\lra} \tbwd{e} R^e \hra \big(\tbwd{e} R^e\big) \oxr Q \simeq \tbwd{e} (R^e \oxr Q) \simeq \tbwd{e} Q^e \simeq Q. $$
Therefore, we get commutative diagrams
\begin{equation} \label{diag}
\xymatrix{
\bwd^e E \ar[rr]^{\bwd^e f} \ar[d] \ar@{->}@/_{4pc}/[dddd]_{\rho} & & \bwd^eR^e \ar[d] \ar@{->}@/^{4.2pc}/[dddd]^{\gamma}\\
(\bwd^e E) \oxr Q\ar[rr]^{(\bwd^e f) \ox \Id} \ar[d]_{\simeq} & & (\bwd^eR^e) \oxr Q \ar[d]^{\simeq}\\
\bwd^e (E \oxr Q) \ar@{->}[rr]^{\bwd^e(f \ox \Id)} \ar[d]_{\simeq} & & \bwd^e(R^e \oxr Q) \ar[d]^{\simeq}\\
\bwd^e Q^e \ar[rr] \ar[d]_{\simeq} & & \bwd^eQ^e \ar[d]^{\simeq}\\
Q \ar[rr]  & & Q
}
\end{equation}
where $(\bwd^e f) \ox \Id$ is an isomorphism. This implies that $\ker(\bwd^e f) \subseteq \tau_{R}(\bwd^{e}E)$.
Moreover, $\bwd^{e}R^{e}$ is torsionfree and so $\ker(\bwd^e f)=\tau_{R}(\bwd^e E)$ is independent of $f$. Therefore one has
\begin{equation}\label{e1}
\deter_f(E) \simeq \tbwd{e} E / \tau_{R}(\tbwd{e} E).
\end{equation}
This independence allow us to forget the homomorphism $f$. In fact, in \cite{wv2} the ideal $\deter_{f}(E)$ is denoted by $\deter_0(E)$. This ideal is, then, an invariant of $E$, called the {\it order determinant of $E$}, and we have
\begin{equation} \label{allideals0}
\deter_0(E):=\deter_f(E),
\end{equation}
where $\apl{f}{E}{R^{e}}$ is an $R$-homomorphism with $\rank f=e$.
Therefore
\begin{equation} \label{allideals}
\deter_0(E)=\deter_f(E)=F_0(G/\im f) = \im \tbwd{e} f \simeq \tbwd{e} E / \tau_{R}(\tbwd{e} E).
\end{equation}
(To be more precise, we should define $\deter_0(E)$ as the class of fractional ideals isomorphic to $\tbwd{e} E / \tau_{R}(\tbwd{e} E)$, but for the purposes in this paper we prefer to think $\deter_0(E)$ as any ideal of $R$ in this class of the form $\deter_f(E)$, with $f$ of rank $e$. Note that since $\deter_f(E) = I_e(\psi)$ with $\rank \psi = e$, then $\deter_f(E)$ has positive grade.)

\smallskip

Assume that $E$ is a \fgtfrmo having rank $e >0$. The modules $\bwd^{e}E$ and $\bwd^{e}E/\tau_{R}(\bwd^{e}E)$ have the same Rees algebra. Therefore, the above approach allows us, by means of the equalities in Eq. (\ref{allideals}), to deduce a lower bound for the ana\-ly\-tic spread $\ell\big(\bwd^{e}E\big)$. In particular, we obtain an affirmative answer for ideal modules to the  question posed in \cite{su}. We note that ideal modules are orientable (which means that $\big(\bwd^{e}E\big)^{**}\simeq R$); in this case an affirmative answer is given in \cite[Corollary~3.2]{su} under the assumption that $R$ is a local equidimensional and universally catenary ring.

\begin{proposition}\label{zak}
Let $R$ be a \nol{} and $E \sube G \simeq R^e$ a
   \fgtfrmo having rank $e >0$, but not free. Then
$$ \ell\big(\tbwd{e}E\big) \geq \Ht F_{0}(G/E). $$
If $E$ is an ideal module then
$$ \ell\big(\tbwd{e}E\big) \geq \Ht F_{e}(E). $$
\end{proposition}

\begin{proof}
In this case $\im f = E$ and we have
\begin{align*}
\ell\big(\tbwd{e}E\big) & = \ell\big(\tbwd{e}E\big/\tau_{R}\big(\tbwd{e}E\big)\big) = \ell(\deter_{0}(E)) = \ell(F_0(G/E)) \\
& \geq \Ht(F_0(G/E)) = \Ht F_{e}(E)
\end{align*}
- the last equality holds in the case where $E$ is an ideal module (by (\ref{FeF0})).
\end{proof}

\subsection{The norm of a module}

Let $E$ be a finitely generated $R$-module having rank $e$. O. Villamayor attached in \cite{ov} a class of ideals to $E$ called the norm of $E$ and denoted by $[[E]]_R$, and showed that the blow up at any representant of this class has a universal flattening property.
According to \cite{ov}
\begin{equation}\label{e2}
[[E]]_{R}:= \im \big(\tbwd{e} E \ra Q \simeq \tbwd{e}E \oxr Q \big)
\end{equation}
and, any fractional ideal isomorphic to this one is a representative of $[[E]]_{R}$. In particular, $[[E]]_{R}= \im \rho$, where $\rho$ is as in the diagram (\ref{diag}). Moreover, since $\im \rho \simeq \gamma(\im  \bwd^e f)$ then considered $\deter_{0}(E)$ as fractional ideal we get
\begin{equation}\label{e3}
[[E]]_{R} \simeq \deter_{0}(E).
\end{equation}

Let $X=\spec(R)$. Then the blow up of $X$ at $E$, denoted by $\Bl_{E}(X)$, is defined as the blow up of $X$ with respect to any representative of $[[E]]_{R}$. In particular, we have that
$$ \Bl_{E}(X) \simeq \Proj \CR[\deter_{0}(E)t] = \Proj\left(\bop_{n\geq 0}\deter_{0}(E)^{n}t^{n}\right).$$

The following construction is done in \cite{ov} to determine another ideal representing $[[E]]_{R}$: Since $E$ has rank $e$ we can choose an $n \x (n-e)$ submatrix $\varphi'$ of  $\varphi$, in (\ref{pE}), with rank $(n-e)$, that is having a non-zero divisor $(n-e)\x(n-e)$ minor. Hence there exists a free $R$-submodule $M$ of $\ker \phi$ having rank $(n-e)$. The module
$E_1:=R^{n}/M = \coker(R^{n-e} \overset{\varphi'}{\ra} R^{n})$
is an \rmo{} with a finite free presentation $0 \ra R^{n-e} \ra R^n \ra E_{1} \ra 0$. Moreover, the module $E_{1}$ has rank $e$, there is a natural surjection $E_{1} \overset{\nu}{\tra} E$ and $E_{1}/\tau_{R}(E_{1}) \simeq E/\tau_{R}(E)$. Further, by Properties \ref{dfprop},
$$ [[E]]_{R} \simeq \deter_{0}(E) = \deter_{0}(E/\tau_{R}(E)) = \deter_{0}(E_{1}/\tau_{R}(E_{1})) = \deter_{0}(E_{1}) \simeq [[E_{1}]]_{R}.$$
Therefore, the blow ups at $E$ and $E_{1}$ are isomorphic. Moreover, by \cite[Proposition~2.5]{ov}, the Fitting ideal $F_{e}(E_{1})$ is a representative of $[[E_{1}]]_{R}$ as a fractional ideal. Hence
$$ \Bl_{E}(X) \simeq \Proj \CR[F_{e}(E_{1})t] = \Proj\left(\bop_{n\geq 0}F_{e}(E_{1})^{n}t^{n}\right).$$
Note that the surjection $\nu$ implies an inclusion of the respective Fitting ideals, and so we have
\begin{equation} \label{e4}
[[E]]_{R} \simeq [[E_{1}]]_{R} \simeq F_{e}(E_{1}) \subseteq F_{e}(E).
\end{equation}

Using (\ref{e4}) we obtain another inequality for the analytic spread of $\tbwd{e}E$.

\begin{proposition}\label{zak2}
Let $R$ be a \nol{} and $E$ a
   \fgrmo{} having rank $e >0$, but not free. Then, there exists a module $E_1$ of projective dimension $1$ having rank $e$ such that $E_{1}/\tau_{R}(E_{1}) \simeq E$, and
$$ \ell\big(\tbwd{e}E\big) \geq \Ht F_{e}(E_{1}). $$
\end{proposition}

\begin{proof}
Let $E_1$ constructed as above. Then
\begin{align*}
\ell\big(\tbwd{e}E\big) & = \ell\big(\tbwd{e}E\big/\tau_{R}\big(\tbwd{e}E\big)\big) = \ell(\deter_{0}(E)) = \ell([[E]]_{R}) \\
& = \ell(F_{e}(E_{1})) \geq \Ht F_{e}(E_{1}),
\end{align*}
as asserted.
\end{proof}

In particular, if $\pjd E\leq 1$ then $E=E_{1}$, and we get another affirmative answer for the problem mentioned in \cite{su}, without additional assumptions on $R$.

\begin{corollary}\label{zak3}
Let $R$ be a \nol{} and $E$ a \fgr{} having rank $e$ with $\pjd E= 1$. Then
$$ \ell\big(\tbwd{e}E\big) \geq \Ht F_{e}(E). $$
\end{corollary}

The module $E_1$ is also useful to prove the equality of the first non-zero Fitting ideals of $E$ and $G/E$, for ideal modules having projective dimension equal to one. We note that this equality was proved in \cite{hu} in the case where $R$ is a $2$-dimensional regular local ring.

\begin{proposition} \label{idmpj1}
Let $R$ be a \nor{} and $E \subseteq G \simeq R^e$ a \fgtfrmo having rank $e >0$. Then, up to isomorphism,
$$ F_0(G/E) \subseteq F_{e}(E). $$
Moreover, if  $E$ is an ideal module with $\pjd E=1$ then
$$  F_0(G/E)=F_e(E).$$
\end{proposition}

\begin{proof}
Let $E_1$  constructed as above. Using Eq. (\ref{iguald}), (\ref{e3}) and (\ref{e4}) we get
$$ F_0(G/E) = \deter_0(E) \simeq [[E]]_{R} \simeq F_{e}(E_{1}) \subseteq F_{e}(E). $$
Now, if $\pjd E =1$ then $F_{e}(E_{1}) = F_{e}(E)$. Moreover, if $E$ is an ideal module then
$$ \grd F_0(G/E)=\grd F_e(E) = \grd G/E \geq 2, $$
and so $F_0(G/E)=F_e(E)$, as asserted.
\end{proof}

\subsection{The universal property of the blow up at a module}

Next we show that the blow up at the order determinant $\deter_{0}(E)$ has an universal flattening property. This gives an alternative to \cite[Theorem 3.3]{ov}.

\smallskip

To do this we begin to prove an easy consequence of Lipman's Theorem (see \cite[Theorem~D.18]{kunz2}).

\begin{proposition} \label{lip}
Let $R$ be a \nol{} and $E \subseteq G \simeq R^e$ a \fgtfrmo having rank $e >0$. Then the following are equivalent:
\begin{enumerate}
\item $E$ is free.
\item $F_0(G/E)$ is a principal ideal (generated by a non-zero divisor of $R$).
\item $\deter_0(E)$ is invertible.
\end{enumerate}
\end{proposition}

\begin{proof}
``(a) $\Rightarrow$ (b)'' If $E$ is free, we have a natural exact sequence $0 \ra E \ra G \ra G/E \ra 0$ with $E$ and $G$ free $R$-modules. Since $G/E$ has rank $0$ then $\tau_R(G/E)=G/E$ and so (b) follows by Lipman's Theorem.

``(b) $\Leftrightarrow$ (c)'' is clear.

``(b) $\Rightarrow$ (a)'' By Lipman's Theorem, $\pjd G/E \leq 1$, and so $\pjd E =0$ proving that $E$ is free.
\end{proof}

 Under our conditions, the ideal $\deter_0(E)$ behaves well by localization as well as the computation of the inverse. Therefore the ideal $\deter_0(E) \deter_0(E)^{-1}$ defines the non-free locus of $E$.

\begin{corollary} \label{lip2}
Let $R$ be a \nor{} and $E \subseteq G \simeq R^e$ a \fgtfrmo having rank $e >0$. Then
$$ V(\deter_0(E) \deter_0(E)^{-1}) = \spec(R) \setminus V(F_e(E)).$$
\end{corollary}

\begin{proof}
Let $\frp$ be an $R$-prime ideal. We have
\begin{align*}
E_{\frp} \text{ is free } & \iff \deter_0(E_{\frp}) = \deter_0(E)_{\frp} \text{ is invertible}  \\
& \iff  (\deter_0(E)_{\frp}) (\deter_0(E)_{\frp})^{-1} = R_{\frp} \\
& \iff  (\deter_0(E) \deter_0(E)^{-1})_{\frp} = R_{\frp} \\
& \iff \frp \nsupseteq  (\deter_0(E))(\deter_0(E))^{-1},
\end{align*}
and the equality follows.
\end{proof}

We say that a ring homomorphism $\apl{\varphi}{R}{S}$ is {\it birational} if $\varphi$ induces an isomorphim $\Q(R) \simeq \Q(S)$ between the total ring of fractions.

\begin{proposition} \label{bir}
Let $R$ be a \nor{} and $E$ a \fg{} \rmo{} having rank $e$. Let $\apl{\varphi}{R}{S}$ be a birational homomorphism of rings.
Then
$$ E\oxr S /\tau_{R}(E \oxr S) \text{ is free } \iff \deter_{0}(E) \cdot S \text{ is invertible}. $$
\end{proposition}

\begin{proof}
Since $E$ has rank $e$ there exists an $R$-homomorphism $\apl{f}{E}{R^{e}}$ with $\rank f=e$.
By Properties \ref{dfprop}
\begin{align*}
\deter_{\ol{f \ox \Id}}(E\oxr S /\tau_{R}(E \oxr S)) &=\deter_{f \ox \Id}(E\oxr S) = \deter_{f}(E) \cdot S.
\end{align*}
Now we observe that $E \oxr S$ has rank $e$. In fact,
since $\varphi$ is birational
\begin{align*}
(E \oxr S) \ox_{S} \Q(S) & \simeq E \oxr \Q(S) \simeq E \oxr \Q(R) \ox_{\Q(R)}\Q(S)\\ & \simeq \Q(R)^{e} \ox_{\Q(R)}\Q(S) \simeq \Q(S)^{e}.
\end{align*}
Moreover, we also have
\begin{align*}
\im(\ol{f \ox \Id}) \ox_{S} \Q(S) & = (\im f \oxr S) \ox_{S}\Q(S) \simeq \Q(S)^{e}.
\end{align*}
Therefore $f \ox \Id$ has rank $e$ and, by Lemma \ref{lem1}, $\ol{f \ox \Id}$ is an embedding. Now the result follows by Proposition \ref{lip}.
\end{proof}

Following the notation of \cite{ov}, for a given scheme $(X,\cO_{X})$ we denote by $\ol{\Q(X)}$ the sheaf of total quotient rings of $\cO_{X}$. Then, for a given sheaf $\cE$ of $\cO_{X}$-modules, we also denote by $\tor(\cE)$ the subsheaf of torsion of $\cE$, so that $\cE/\tor(\cE)$ is a sheaf of torsionfree $\cO_{X}$-modules.

\begin{theorem}\label{blow}
Let $R$ be a reduced \nor{}, $E$ a \fgrmo{} having positive rank $e$ and $\deter_{0}(E)$ the order determinant of $E$. Let $X \overset{\pi}{\lra}\spec(R)$ be the blow up of $R$ with center $\deter_{0}(E)$. Then
\begin{enumerate}
\item $\pi^{*}(E)/\tor(\pi^{*}(E))$ is a locally free sheaf of $\cO_{X}$-modules of rank $e$.
\item For any birational morphism $Y \overset{\gamma}{\lra} \spec(R)$ for which $\gamma^{*}(E)/\tor(\gamma^{*}(E))$ is locally free of rank $e$, there is a unique morphism $\apl{\beta}{Y}{X}$ such that $\beta \circ \pi=\gamma$.
\end{enumerate}
\end{theorem}

\begin{proof}
The proof  is direct, just taking into account that under our conditions the blow up is a birational morphism, the universal property of the blow up, Proposition \ref{bir}, and that we only need to check it locally.
\end{proof}

\section{Reduction of modules}\label{red}

\noindent
We briefly recall the definition of reduction of a module, and notions around, and establish some basic properties that we need in the next section.

\smallbreak

Suppose that $E$ is a \fgrmo{} having a rank over a
\nor{} $R$. We denote by $\RE$ the Rees algebra of $E$, which is
by definition the symmetric algebra module its torsion
$R$-submodule. An \rmo{} $U$ of $E$ is said to be a {\it reduction} of $E$ if
$$ \RE_{r+1} = (U/\tau_{R} (U)) \cdot \RE_{r} $$
for some $r \geq 0$. Since $E$ and $E/\tau_{R}(E)$ have the same rank and the same Rees algebra, then $U$ is a reduction of $E$ if and only if $U/\tau_{R}(U)$ is a reduction of $E/\tau_{R}(E)$. Hence, we may often assume that $E$ is torsionfree, since the theory of reductions for torsionfree modules affords some additional properties.
The least integer $r$ for which $\RE_{r+1}= (U/\tau_{R}(U))
\cdot \RE_{r}$ is called {\it the reduction number of} $E$ {\it
with respect to} $U$, and is denoted by $r_{U}(E)$.

\smallbreak

The Rees algebra does not commute in general with the extension of scalars. A simple example was given by A. Micali in \cite[Chap. I, $\S 5$, Example]{mi}: Let $(R, \fm)$ be a local domain and let $k = R/\fm$. Then, $\CR_R(k)\otimes _R k \simeq R \otimes _R k \simeq k$ while $\CR_k(k\otimes _R k) \simeq k[t]$. Nevertheless, because the symmetric algebra always commutes with the extension of scalars, the Rees algebra commutes if the torsion also commutes. In particular, we have that the Rees algebra commutes with polynomial and Nagata extensions, which implies that reductions are preserved by polynomial and Nagata extensions.

\smallbreak

Let $E$ be a \fgtfrmo having rank. A reduction of $E$ is called {\it minimal} if it is minimal with respect to
inclusion. Minimal reductions always exist. The {\it reduction
number} of $E$, denoted by $r(E)$, is the minimum of $r_{U}(E)$,
where $U$ ranges over all minimal reductions of $E$.

\smallbreak


Over an integral domain the order determinants can be used to check whether a submodule can be a reduction of a module (cf. \cite[Proposition~8.66]{wv2}).
Since, under the assumption that $\pjd E=1$, we have
\begin{equation}\label{pjd1}
\deter_{0}(E)=F_{e}(E),
\end{equation}
then we may assert the following:

\begin{proposition}
Let $R$ be an integral domain and let $U$ and $E$ be \fgtfrmos of rank $e$ with $U \subset E$. Suppose that $\pjd E=\pjd U=1$. Then $U$ is a reduction of $E$ \sse $F_{e}(U)$ is a reduction of $F_{e}(E)$.
\end{proposition}

\smallbreak

In \cite[Proposition~2.1]{cz2} or \cite[Theorem~3.3]{oo} we can find the fundamental properties of reduction of modules. For completeness we prove that the minimal number of generators of a \fgtfrmo $E$ and of any minimal reduction are well related as in the case of ideals.


\begin{lemma}\label{mred}
    Let $\Rm$ be a \nol, $E$ a \fgtfrmo and  $U$ an
    $R$-sub\-mo\-du\-le of $E$.
    Then $U$ is a reduction of $E$ \sse $U+\fm E$ is a
    reduction of $E$.
\end{lemma}

\begin{proof}
    Suppose that $U$ is a reduction of $E$. Then $\RE_{r+1} = U \cdot
    \RE_{r}$ for some $r \geq 0$ and we have
    $$ (U+\fm E) \cdot \RE_{r} = U \cdot \RE_{r} + (\fm E) \cdot
    \RE_{r} = \RE_{r+1}.$$

    Conversely, assume that $(U+\fm E) \cdot \RE_{r} = \RE_{r+1}$ for
    some $r \geq 0$. Then
    $$ \RE_{r+1} = (U+\fm E) \cdot \RE_{r}= U \cdot \RE_{r} + (\fm
    E) \cdot \RE_{r} = U \cdot \RE_{r} + \fm \RE_{r+1} $$
    and so, by Nakayama's lemma, $\RE_{r+1} = U \cdot
    \RE_{r}$.
\end{proof}

\begin{proposition}\label{nota}
    Let $\Rmk$ be a \nol{} and $E$ a \fgtfrmo having rank. If $U=
    Ra_{1} + \cdots + Ra_{n}$ is a minimal reduction of
    $E$ with $n= \mu(U)$, then $\sek{\ol{a}}{n}$ are linearly
    independent over $\frk$, where $\ol{a}_{i}=a_{i}+\fm E$. In particular  $\sek{a}{n}$
    are part of a minimal system of generators of $E$. Moreover,
    $U \cap \fm E = \fm U$ and $\mu(E) = \mu(U) + \mu(E/U)$.
\end{proposition}

\begin{proof}
    In fact, if $\sek{\ol{a}}{n}$ are not linearly
    independent over $\frk$, without loss of ge\-ne\-ra\-li\-ty we may
    assume that $\ol{a}_{n} \in \langle  \sek{\ol{a}}{n-1}
    \rangle$, so $V = Ra_{1} + \cdots + Ra_{n-1} \varsubsetneq
    U$ is also a reduction of $E$
    (by \refl{mred}), contradicting the minimality of $U$. Hence
    $\sek{\ol{a}}{n}$ are \li{} over $\frk$ and so
    $U \cap \fm E = \fm U$.
    This implies that any minimal generating set of $U$ can be
    extended to a minimal generating set of $E$. In particular,
    $\sek{a}{n}$ are part of a minimal system of generators
    of $E$. Furthermore,
    $$ 0 \ra U/\fm U \simeq (\fm E + U) /\fm E \ra E/\fm E \ra
    E/(\fm E + U) \simeq (E/U)/\fm (E/U) \ra 0 $$
    is an exact sequence. Hence
    $$   \mu(E) = \dim_{\frk}(E/\fm E) = \dim_{\frk}(U/\fm U) +
    \dim_{\frk}((E/U)/\fm (E/U)) = \mu(U) + \mu(E/U), $$
    and the result follows.
\end{proof}

\begin{corollary}\label{nota1}
    Let $\Rm$ be a \nol{} and $E$ a \fgrmo{} having positive
    rank. If $U$ is a reduction of $E$ then $U \nsub \fm E$.
\end{corollary}

\begin{proof}
Since if $U \subseteq \fm E$ then $U/\tau_{R}(U) \subseteq \fm (E/\tau_{R}(E))$, we may assume that $E$ is torsionfree.  Let $V \sube U$ be a minimal reduction of $E$. We begin to prove
    that $V \nsub \fm E$. If not, since $V \cap \fm E = \fm V$ (by
    \refp{nota}) we
    deduce that $V= \fm V$. Hence, by Nakayama's Lemma, $V=0$,
    a contradiction.
    Therefore $V \nsub \fm E$, and so $U \nsub \fm E$.
\end{proof}

Any reduction $U$ of $E$ has rank and $\rank U = \rank E$.
Moreover, $E/U$ is a torsion module and $\grd E/U >0$. Further, if  $V \sub U$ with $E/V$ having rank then $U/V$ is a reduction of $E/V$ (see \cite[Proposition~2.4]{cz3}).
\smallbreak

Given a \nol{} $\Rmk$ and a \fgtfrmo $E$ having rank,
it is known that
\begin{equation}\label{mule}
\mu(U) \geq \lE = \dim \FE
\end{equation}
for any reduction $U$ of $E$ with equality \sse $U$ is a minimal reduction of $E$, supposing $k$ infinite. In fact, in this case the classes in $E/\fm E$ of any minimal generating set of $U$ are a system of parameters of $\FE= \RE \oxr R/\fm$.

\smallbreak
\section{The quotient module $\ol{E}=E''/F$}\label{qm}

\noindent
In this section we recall the construction of a \gbi{} given in
\cite{suv}. We explore the inductive process of this construction
in order to prove a formula for the minimal number of generators of a
\gbi. We shall do this more generally for a quotient module
$\ol{E}$ of rank $1$ (not necessarily an ideal).
This result will be applied in the next section mainly in the case of modules with projective dimension one.

\smallbreak

Let $\Rm$ be a \nol, and $E$ a \fgrmo{} having a rank $e > 0$, and
let $U = \sum_{i=1}^n Ra_{i}$ be a submodule of $E$ such that $E/U$
is a torsion module (which holds if $U$ is a reduction of $E$). Further, let
$\bZ = \{ z_{ij} \mid 1 \leq i \leq n, 1 \leq j \leq e-1 \}$
be a set of $n \times (e-1)$ indeterminates over $R$. We fix the
notation
$$ R\pr = R[z_{ij} \mid 1\leq i \leq n,\ 1\leq j \leq e-1]=R[\bZ] \, , \;
R\prd = R\pr_{\fm R\pr}=R(\bZ). $$
The ring $R\prd$ is a local ring with maximal ideal $\fm R\prd$ and
infinite residue field $k(\bZ)$ and is called the Nagata extension of $R$
\wrt $\bZ$. Further, we set
$$ U\pr = U \oxr R\pr \, , \;U\prd = U \oxr R\prd \, , \;E\pr =
E \oxr R\pr \, , \;E\prd = E \oxr R\prd \,, $$
$$ x_{j} = \sum_{i=1}^n z_{ij} a_{i} \in U\pr \sube E\pr \, ,
\;F = \sum_{j=1}^{e-1} R\prd x_{j}. $$
In \cite{suv} it is proved that $F$ is a free module over
$R\prd$ of rank $e-1$.
We denote the $R\prd$-module $E\prd/F$ by
$\ol{E}$; this module has rank 1. There are good relations between some numerical invariants of $\ol{E}$ and $E$. In particular, the Krull dimensions of their Rees algebras and of their fiber cones are related by the equalities
\begin{align}
& \dim \RE = \dim \CR(\ol{E})+e-1, \label{Ebarra}\\
& \ell(E) = \ell(\ol{E}) +e -1. \label{Ebarra1}
\end{align}

\smallbreak

In \cite{cz3} we explored the inductive process of this construction and made deeper the relations between $E$ and $\ol{E}$. In particular, we proved that, under certain conditions, the generators $\sek{x}{e-1}$ of $F$ form a regular sequence on $\CR(E'')$ and
\begin{equation} \label{Ebarra2}
\CR(\ol{E}) \simeq \CR(E'')/\langle x_{1},\ldots,x_{e-1}\rangle
\end{equation}
(cf \cite[Theorem~3.7]{cz3}), hence $\CR(E'')$ is a deformation of $\CR(\ol{E})$. As a consequence we got
\begin{equation} \label{Ebarra3}
\dpt \RE = \dpt \CR(\ol{E}) +e -1.
\end{equation}

\smallbreak

In this section, following our approach in \cite{cz3}, we prove that
\begin{equation}\label{Ebarra4}
\mu(\ol{E}) = \mu(E) +e -1.
\end{equation}
This equality allow us to construct a minimal generating set $\sek{x}{e-1},\ldots,x_{n}$ of $E''$ containing the generators $\sek{x}{e-1}$ of $F$.

We use the same notation as in \cite{cz3}.
For $1\leq j \leq e-1$ set
$$\bZ_{j} = \{ z_{1j},\ldots,z_{nj} \}$$
and
$$R_{j}= R[\sek{\bZ}{j}] \; , \; R_{j}\prd = {R_{j}}_{\fm R_{j}}=
R(\sek{\bZ}{j}).$$
Then
$$ R_{j} = R_{j-1} \oxr R[\bZ_{j}]=R_{j-1}[\bZ_{j}]. $$
According to \cite[Lemma~3.1]{cz3}, we have
  $$R\pr \simeq R_{j} \oxr R[\bZ_{j+1},\ldots, \bZ_{e-1}] \simeq
R_{j}[\bZ_{j+1},\ldots, \bZ_{e-1}],$$
$$ R\prd \simeq \left( R_{j}\prd[\bZ_{j+1},\ldots, \bZ_{e-1}] \right)_{\fm
R_{j}\prd[\bZ_{j+1},\ldots,\bZ_{e-1}]} = R_{j}\prd(\bZ_{j+1},\ldots, \bZ_{e-1}), $$
and
\begin{align*}
R_{j}\prd & \simeq (R_{j-1}\prd[\bZ_{j}])_{\fm R_{j-1}\prd[\bZ_{j}]} = R_{j-1}\prd(\bZ_{j}) = R(\sek{\bZ}{j-1})(\bZ_{j}) \\
& \simeq R(\bZ_{j})(\bZ_{1}, \ldots, \bZ_{j-1}) = \left( (R[\bZ_{j}]_{\fm R[\bZ_{j}]}) [\bZ_{1}, \ldots , \bZ_{j-1}]\right)_{\fm (R[\bZ_{j}]_{\fm R[\bZ_{j}]})[\bZ_{1}, \ldots , \bZ_{j-1}]}.
\end{align*}
Moreover, for $1\leq j \leq e-1$, set $E_{j} = E \oxr R_{j}$, $E_{j}\prd
= E \oxr R_{j}\prd \simeq E_{j} \ox_{R_{j}} R_{j}\prd$.
Hence
$$ E_{j} = E_{j-1} \oxr R[\bZ_{j}] = E_{j-1} \ox_{R_{j-1}}R_{j}, $$
$$ E\pr \simeq E_{j}\ox_{R_{j}} R\pr \simeq E_{j} \oxr
R[\bZ_{j+1},\ldots,\bZ_{e-1}], $$
$$ E\prd \simeq E\pr \ox_{R\pr} R\prd \simeq E_{j}
\ox_{R_{j}}R\prd \simeq E_{j}\prd \ox_{R_{j}\prd}R\prd, $$
$$ E_{j}\prd \simeq E_{j-1}\prd \ox_{R_{j-1}\prd}R_{j}\prd \; , \;
\ol{E_{j}\prd} \simeq \ol{E_{j}} \ox_{R_{j}}R_{j}\prd. $$

Furthermore, for $1\leq j \leq e-1$, set $\ol{E_{j}} =
E_{j}/\seg{x}{j}$, $\ol{E_{j}\prd} =
E_{j}\prd/\seg{x}{j}$,
where $\seg{x}{j}$ denote in each case the sub\-mo\-du\-le generated
by $\sek{x}{j}$. By convention $\ol{E_{0}} = E= E_{0}$, $R_{0}=R$. We have, for each $1\leq j \leq e-1$,
\begin{equation} \label{ejq}
\ol{E_{j}} \simeq (\ol{E_{j-1}} \ox_{R_{j-1}}R_{j})/\langle \ol{x_{j}} \rangle \; , \; \ol{E_{j}\prd} \simeq (\ol{E_{j-1}\prd} \ox_{R_{j-1}\prd}R_{j}\prd)/\langle \ol{x_{j}} \rangle,
\end{equation}
where $\ol{x_{j}} = x_{j} + \seg{x}{j-1}$ in each case (cf. \cite[Lemma~3.2]{cz3}).
These relations for $E$ are also true for any submodule $U$
of $E$.
Moreover, since $R\prd$ is the Nagata extension of $R_{j}\prd$ \wrt $\bZ_{j+1},
\ldots, \bZ_{e-1}$,
\begin{equation} \label{ejq2}
\rank \ol{E_{j}\prd} = \rank \ol{U_{j}\prd}=e-j \geq 2 \;\, , \; \,
\mu(E_{j}\prd)=\mu(E\prd)
\end{equation}
(cf. \cite[Lemma~3.3]{cz3}).

\medbreak

For the module $\ol{E}$ constructed above we are able to
prove that $\mu(\ol{E}) = \mu(E)-e+1$. We argue by induction on the rank of $F$.

\begin{lemma}\label{mu}
     Let $\Rmk$ be a \nol, $E$ a \fgrmo{} having
     rank $e \geq 2$ and $U=\sum_{i=1}^nRa_{i}$ a reduction of $E$.
     Let $R\pr=R[\sek{z}{n}]$, $R\prd={R\pr}_{\fm R\pr}$, $x =
     z_{1}a_{1} + \cdots + z_{n} a_{n}$
     and let
     $U\prd = U \oxr R\prd$, $E\prd= E \oxr R\prd$. Then
     $\mu(V\prd/R\prd x) = \mu(V) - 1$,
     for every $R$-submodule $V$ of $E$ containing $U$ and
     $V\prd = V \oxr R\prd$.
\end{lemma}

\begin{proof}
    Let $V$ be any \rmo{} such that $U \sube V \sube E$. Since $U
    =\sum_{i=1}^nRa_{i} \nsub \fm E$ (by \refc{nota1}) there
    exists an $i$ such that $a_{i} \not \in \fm E$. \Wlog
    we may assume that $a_{1} \not \in \fm E$.
    We claim that $x/1 \in U\prd \sube V\prd \sube E\prd$
    is part of a minimal generating set of
    $V\prd$. Suppose not, that is $x/1 \in \fm V\prd$.
    Since $V\prd = V\pr
    \ox_{R\pr} {R\pr}_{\fm R\pr} = {V\pr}_{\fm R\pr}$ there exists
    $y \in R\pr \setminus \fm R\pr$ such that $y x \in \fm V\pr$.
    Set
    $$\ol{y} = y + \fm R\pr \in R\pr/\fm R\pr \simeq R/\fm \oxr R\pr
    \simeq k[\sek{z}{n}], $$
    and set
    $$\ol{x} = x + \fm V\pr = \ol{a_{1}} z_{1} + \cdots + \ol{a_{n}}
    z_{n} \in V\pr/\fm V\pr = V/\fm V \oxr
    R\pr = (V/\fm V)[\sek{z}{n}], $$
    with $\ol{a_{i}}= a_{i} + \fm V \in V/\fm V$, $(i=1,\ldots,n)$.
    Hence $\ol{y}\, \ol{x} = 0$ with $\ol{y} \neq 0$.
    Now write
    $\ol{y} = \sum_{j=0}^m \ol{y_{j}} z_{1}^j$
    with $\ol{y_{j}} \in k[z_{2},\ldots,z_{n}]$, $(j=0,\ldots,m)$
    and  $\ol{y_{m}} \neq 0$.
    Hence, we have
        $$ \ol{y} \, \ol{x} = \sum_{i=1}^n \ol{y} \,\ol{a_{i}}
        z_{i} = \ol{y} \,\ol{a_{1}}z_{1} + \sum_{i=2}^n \ol{y} \,
        \ol{a_{i}}z_{i}
     = \sum_{j=0}^m \ol{y_{j}} \, \ol{a_{1}}z_{1}^{j+1} +
    \sum_{i=2}^n \ol{y} \,\ol{a_{i}}z_{i} $$
    and since $\ol{y}\, \ol{x} = 0$ we must have
    $\ol{y_{m}} \, \ol{a_{1}} z_{1}^{m+1} =0$. Hence
    $\ol{y_{m}}\, \ol{a_{1}} =0$.
    But $\ol{y_{m}}\, \ol{a_{1}} \in$ \linebreak $V/\fm
    V[z_{2},\ldots,
    z_{n}] = k[z_{2},\ldots,z_{n}] \ox_{k} V/\fm V$ with
    $k[z_{2},\ldots,z_{n}]$ \ff{} over $k$ and
    $0 \neq \ol{y_{m}} \in k[z_{2},\ldots,z_{n}]$, $\ol{a_{1}} =
    a_{i} + \fm V \neq 0$
    (because $a_{1} \not \in \fm V \sube \fm E$).
    Therefore $x/1 \not \in \fm V\prd$
    and $x/1$ is part of a minimal generating set of $V\prd$, as
    claimed.
    It follows that
    $$\mu(V\prd/R\prd x) = \mu(V\prd) - 1 = \mu(V) - 1,$$
    proving the assertion.
\end{proof}


\begin{lemma}\label{mu1}
     Let $\Rmk$ be a \nol, let $E$ be a \fg{} \rmo{} having
     rank $e \geq 2$ and $U=\sum_{i=1}^nRa_{i}$ a reduction of $E$.
     Let $R\prd$, $\sek{x}{e-1}$, $F$ be defined as before.
     Set, for each $1 \leq t \leq e-1$, $F_{t} = \sum_{i=1}^t R\prd
     x_{i}$. Then
     \begin{enumerate}
         \item  $\mu(V\prd/F_{t}) = \mu(V) - t$,

         \item  $F_{t} \cap \fm R\prd V\prd = \fm R\prd F_{t}$,
     \end{enumerate}
     for every $R$-submodule $V$ of $E$ containing $U$,
     $V\prd = V \oxr R\prd$.
\end{lemma}

\begin{proof}
Let $V$ be any \rmo{} such that $U \sube V \sube E$ and set $V\prd=V \oxr R\prd$.

a) We use induction on $t$. For $t=1$ we apply the previous lemma, since $R_{1} = R[z_{11},\ldots,$ $z_{n1}]$ and $F_{1}$ is freely generated by $x_{1}= z_{11}a_{1}+ \cdots+z_{n1}a_{n}$.
Suppose that $1 \leq j \leq t \leq e-1$.
We set, as before, $E_{j}= E \oxr R_{j}$, $U_{j} =U \oxr R_{j}$ with $R_{j}=R[z_{11},\ldots,z_{1n},\ldots,z_{1j},\ldots,z_{nj}]=R_{j-1}[\bZ_{j}]$, $R_{j}\prd ={R_{j}}_{\fm R_{j}}$ and $E_{j}\prd = E \oxr R_{j}\prd$, $U_{j}\prd = U \oxr R_{j}\prd$.
Moreover, set $V_{j}=V \oxr R_{j}$, $V_{j}\prd = V \oxr R_{j}\prd= V_{j-1}\prd \ox_{R_{j-1}\prd} R_{j}\prd$ and $F_{j}\prd = \sum_{i=1}^j R_{j}\prd x_{i}$.
Since $R_{j}\prd$ is \ff{} over $R$, $F_{j}\prd \sube U_{j}\prd \sube V_{j}\prd \sube E_{j}\prd$.
Moreover, $F_j = F''_j \otimes_{R''_j} R''$ and $V'' = (V \otimes_R R''_j) \otimes_{R''_j} R''=V''_j \otimes_{R''_j} R''.$
Therefore
\begin{equation}\label{eq2}
V''/F_j \simeq V''_j/F''_j \otimes_{R''_j} R''.
\end{equation}
Now, suppose for induction that $j >1$ and that $\mu(V''/F_{j-1})=\mu(V)-(j-1)$. Hence, by Eq. (\ref{eq2}),
$$\mu(V''_{j-1}/F''_{j-1}) = \mu(V''/F_{j-1})=\mu(V)-(j-1)=\mu(V''_{j-1})-(j-1).$$
Moreover, $R_{j}\prd= R_{j-1}\prd(\bZ_{j})$ is the Nagata extension of $R_{j-1}\prd$ (\wrt $\bZ_{j}$) and, as in Eq. (\ref{ejq}), we have
$$ V_{j}\prd/F_{j}\prd \simeq (V_{j-1}\prd/F_{j-1}\prd \ox_{R_{j-1}\prd} R_{j}\prd)/\langle \ol{x_{j}} \rangle, $$
where $\ol{x_{j}}= x_{j} + R_{j}\prd x_{1}+ \cdots + R_{j}\prd x_{j-1}$.
Further, $U''_{j}/F''_{j} \subseteq V''_{j}/F''_{j}$ with $U''_{j}/F''_{j}$ a reduction of $E''_{j}/F''_{j}$ and $\rank(E''_{j}/F''_{j})=e-j\geq 2$ (by Eq. (\ref{ejq2})).
Hence, by the previous lemma and by induction assumption,
$$ \mu(V_{j}\prd/F_{j}\prd) =  \mu(V_{j-1}\prd/F_{j-1}\prd) -1 = \mu(V_{j-1}\prd) - (j-1) -1 = \mu(V_{j-1}\prd)-j. $$
It follows by induction that
$$ \mu(V\prd/F_{t}) = \mu(V_{t}\prd/F_{t}\prd) =\mu(V_{t-1}\prd)-t = \mu(V)-t$$
(since $R\prd$ is \ff{} over $R_{j}\prd$ for all $j$), proving a).

b) is a direct consequence of a).
\end{proof}

\begin{proposition}\label{free3}
    Let $R$ be a \nol, let $E$ be a \fg{} \rmo{} having
    rank $e \geq 2$ and $U$ a reduction of $E$.
    Let $R\prd$, $E\prd$, $\sek{x}{e-1}$, $F$ be defined as before.
    Set, for each $1 \leq t \leq e-1$, $F_{t} = \sum_{i=1}^t R\prd
    x_{i}$. Then
    \begin{enumerate}
        \item  $\mu(E\prd/F_{t})= \mu(E)-t;$

        \item  $E\prd/F_{t}$ is a free $R\prd$-module \sse $E$ is a
        free $R$-module.
    \end{enumerate}
\end{proposition}

\begin{proof}
a) is a particular case of the previous lemma.

b) Consider the natural exact sequence $0 \ra F_{t} \ra
    E\prd \ra E\prd/F_{t} \ra 0$.

    If $E\prd/F_{t}$ is free, then
    $E\prd \simeq F_{t} \+ E\prd/F_{t}$ and $E\prd$ is
    free over $R\prd$. Hence $E$ is free over $R$.
    Conversely, suppose that $E$ is free. Hence
    $\mu(E)=\rank E =e$.
    Thus, by a),
    $$ \mu(E\prd/F_{t}) = \mu(E) -t = e-t = \rank E\prd/F_{t}$$
    and $E\prd/F_{t}$ is free.
\end{proof}

\begin{theorem}\label{free2}
    Let $\Rmk$ be a \nol, let $E$ be a \fg{} \rmo{} having
    rank $e \geq 2$ and $U$ a reduction of $E$.
    Let $\ol{E}= E\prd/F$, $\ol{U} = U\prd/F$
    with $F \sub U\prd \sube E\prd$ as before.
    Then
    \begin{enumerate}
      \item  $\mu(\ol{E})= \mu(E)-e+1;$

      \item  $\ol{E}$ is a free $R\prd$-module \sse $E$ is a free
        $R$-module;

     \item  $\ol{E}$ has finite projective dimension \sse $E$ has finite projective dimension; if this is the case, $\pjd \ol{E} = \pjd E$;

      \item  Assume in addition that $E$ is torsionfree and $k$ is infinite. If $U$ is a minimal reduction of $E$
      then $\ol{U}$ is a minimal reduction of $\ol{E}$.
    \end{enumerate}
\end{theorem}

\begin{proof}
a) and b) are particular cases of the previous result.

c) Consider the exact sequence $0 \ra F \ra E\prd \ra \ol{E} \ra 0$.
Since $F$ is a free $R\prd$-module $\pjd F=0$ and so the first assertion follows.
The second is a consequence of b) and of \cite[Proposition~VII.1.8]{kunz}.

d)  We already observed that $\ol{U}$ is a reduction of $\ol{E}$.
    Now, we have
    $$ \mu(\ol{U}) = \mu(U) -e+1 = \ell(E) -e+1 =
    \ell(\ol{E}),$$
    (by \refl{mu1}, Eq. (\ref{mule}) and (\ref{Ebarra1})), and so $\ol{U}$ is a minimal reduction of $\ol{E}$.
\end{proof}

\section{Generic Bourbaki ideals and Fitting ideals}\label{gbifi}

\noindent
Given a \fgtfrmo $E$ having rank $e > 0$ and satisfying $\widetilde{G_{2}}$ (that is $\mu(E_{\frp})\leq e$ for all $\frp$ with $\dpt R_{\frp}=1$), and given a submodule $U$ of $E$ such that $\grd E/U \geq 2$, there exists a Nagata extension $R\prd=R(\bZ)$ and a free $R\prd$-module $F\simeq {R\prd}^{e-1}$ such that $\ol{E}=E\prd/F$ is torsionfree over $R\prd$ and having rank $1$. In this case $\ol{E} \simeq I_{U}(E)$ is an $R\prd$-ideal (see \cite[Proposition~3.2]{suv}). Such ideal $I=I_{U}(E)$ is called a {\it \gbiu}. If $U=E$ we simply write $I=I(E)$. The construction of a \gbi{} depends on the generating set considered for $U$ and on the set of variables, but is essentially unique.
See \cite[Remark~3.4]{suv}. In particular, if $U$ is a reduction of $E$ and $\grd F_{e}(E) \geq 2$ then there exists $I$ a \gbiu{} with $\grd I >0$ (see \cite[Remark~4.1]{cz3}).

\smallbreak

The aim of this section is to relate generic Bourbaki ideals with Fitting ideals. To do this we first observe that Fitting ideals are related, in a natural way, with the exterior algebra.
Moreover, we observe that whenever we consider a finite presentation $\var$ for $E\prd$, \wrt a generating set containing a basis of $F$ there is a submatrix $\psi$ of $\var$ with the first non-zero Fitting ideal having grade $\geq 1$. This Fitting ideal is isomorphic to a \gbie, and we use it to obtain information about the perfection of the ideal $I(E) \simeq E\prd/F$ (cf. \reft{fi}). This is the main result of this section.

\smallbreak

For any linear map $\apl{\Psi}{R^m}{R^n}$ the Fitting ideal $F_{n-m}(R^{n}/ \im \Psi) = I_{m}(\Psi)$ can be obtained as a image of a given linear map $\bwd^{n-m} R^{n} \ra \bwd^{n} R^{n} \simeq R$. Using the notation introduced in section \ref{div} (in particular Eq. (\ref{xH})), we first note that for complementary subsets $H, K \sube [n]$, we have
\begin{equation}\label{pext}
    x_{H} \wdg x_{K} = (-1)^{\varepsilon_{H,K}} x_{1} \wdg x_{2} \wdg
    \cdots \wdg x_{n},
\end{equation}
where $\varepsilon_{H,K}$ is the number of ordered pairs $(i,j) \in H
\x K$ such that $i >j$.

\begin{lemma}\label{ext3}
    Let $R$ be \nor{} and $\apl{\Psi}{R^m}{R^n}$ a linear map with $n \geq m$.
    Let $\sek{u}{m} \in R^{n}$ be defined by the columns of $\Psi$ and
    consider the map
    $$ \begin{array}{cccl}
        \theta \; : & \bwd^{n-m} R^{n} & \lra & \bwd^{n} R^{n} \\
    & x & \longmapsto & x \wdg u_{1} \wdg \cdots \wdg u_{m}.
    \end{array}$$
    Then $ \im \theta = I_{m}(\Psi). $
\end{lemma}

\begin{proof}
    Suppose
    that $\Psi = (\alp_{ij})$ with respect to a basis $\seq{v}{m}$
    of $R^{m}$ and a basis $\seq{e}{n}$ of $R^{n}$.
    Hence $(e_{H})_{H \in \cP_{n-m}([n])}$ is a basis of $\bwd^{n-m}
    R^{n}$.
    Moreover, since
    $u_{j}= \sum_{i=1}^n \alp_{ij} e_{i}$
    $(j \in [m])$ then by \refeq{pext1},
    $$ u_{1} \wdg \cdots \wdg u_{m} = \sum_{K \in \cP_{m}([n])}
    \deter \Psi_{K,J} \, e_{K}, $$
    where $\Psi_{K,J} = \Psi[\sek{i}{m}|1,\ldots,m]$ for $K=\{
    i_{1} < \cdots < i_{m} \} \sub [n]$, $J = [m]$.
    Therefore, for $H \in \cP_{n-m}([n])$,
    \begin{align*}
        \tet(e_{H}) & = e_{H} \wdg \sum_{K \in \cP_{m}([n])}
    \deter \Psi_{K,J} \, e_{K} \\
    & = \sum_{K \in \cP_{m}([n])} \deter \Psi_{K,J} \,
    (e_{H} \wdg e_{K})\\
    & = \sum_{\substack{K \in \cP_{m}([n])\\K \cap H = \emptyset}}
    \deter\Psi_{K,J} \, (e_{H} \wdg e_{K}) &\text{(since $z \wdg z =0,
    \forall z \in R^{n}$)}\\
    & = \deter \Psi_{[n] \setminus H,J} \, (e_{H} \wdg e_{[n]
    \setminus H})\\
    & = (-1)^{\varepsilon_{[n] \setminus H,J}}
    \deter \Psi_{[n] \setminus H,J} \, (e_{1} \wdg \cdots \wdg e_{n}).
    &\text{(by \refeq{pext})}
    \end{align*}
    Hence $\tet$ is a matrix with one row and $\binom{n}{m}$ columns.
    Therefore,
    $$ \im \tet = \langle \deter \Psi_{[n] \setminus H,J} \rangle_{H
       \in \cP_{n-m}([n])} = I_{m}(\Psi), $$
    as required.
\end{proof}
\smallbreak

In the notation of section \ref{qm}, suppose that $\Rm$ is a
\nol{} and $E$ is a \fgr{} having rank $e \geq 2$.
Suppose also that
$U = \sum_{i=1}^k Ra_{i}$ is a reduction of $E$.  Further set
$R\pr = R[z_{ij} \mid 1 \leq i \leq k,\ 1 \leq j \leq e-1]=R[\bZ]$,
$R\prd = R\pr_{\fm R\pr}=R(\bZ)$ and  $E\prd = E \oxr R\prd$.
Furthermore set
$$ x_{j} = \sum_{i=1}^k z_{ij} a_{i} \; , \; F = \sum_{j=1}^{e-1}
R\prd x_{j} \; , \; \ol{E}=E\prd/F.
$$
We may extend the basis $\seq{x}{e-1}$ of $F$ to a generating set
$\{\sek{x}{n}\}$ of $E\prd$.
We note that, if necessary, we may choose it to be a minimal generating set of $E\prd$ since $\mu(E\prd/F) = \mu(E\prd)-e+1$ (by \reft{free2}).
Let
$$ {R\prd}^m \overset{\varphi}{\ra} {R\prd}^n \overset{h}{\ra} E\prd \ra 0 $$
be a finite free presentation of $E\prd$ \wrt the generators
$\sek{x}{n}$, and write also by $\var$ the respective relation
matrix.
We have
\begin{equation} \label{var}
[x_{1} \,\cdots\, x_{n}] \var = 0,
\end{equation}
and
$$ \rank \var := \rank \im \varphi = \rank K = n-e, $$
where $K= \ker h$ is the module of relations of the generating set $\{
\sek{x}{n} \}$.
Hence $\grd I_{n-e}(\var) \geq 1$ and $I_{n-e+1}(\var)=0$
(by \cite[Proposition~1.4.11]{bh}).
Moreover, the matrix $\var$ has a particular form.

\begin{lemma}\label{phi}
In the above conditions, each row $i \in \{1,\ldots,e-1\}$ of $\var$ is a $Q$-linear combination of the rows $e,\ldots,n$, where $Q = \Quot(R\prd)$. In particular, $\var$ is $Q$-equivalent to
$ \left[ \begin{array}{c}  0 \\ \hline \Psi
    \end{array} \right], $
where $\Psi=\var[e,\ldots,n | 1,\ldots,m]$ is a submatrix of $\var$ satisfying $\grd I_{n-e}(\Psi) \geq 1$.
\end{lemma}

\begin{proof}
    The exact sequence $0 \ra F \ra E\prd \ra \ol{E} \ra 0$ yields the
    exact sequence
    $$ 0 \ra F \ox_{R\prd} Q \ra E\prd \ox_{R\prd} Q \ra \ol{E}
    \ox_{R\prd} Q \ra 0.$$
    Since $\ol{E} \ox_{R\prd} Q$ is free
    $$ E\prd \ox_{R\prd} Q \simeq (F \ox_{R\prd} Q) \+ (\ol{E}
    \ox_{R\prd} Q). $$
    On the other hand, $\ol{\var} = \var \ox_{R\prd} \Id$
    presents $E\prd \ox_{R\prd} Q$ with respect to the generators $x_{1}
    \ox 1, \ldots ,$ $x_{n} \ox 1$ and
    we have the exact sequence
    $$ Q^m \overset{\ol{\var}}{\ra} Q^n \ra (F \ox_{R\prd} Q) \+ (\ol{E}
    \ox_{R\prd} Q) \ra 0. $$
    Moreover, $(x_{1} \ox 1,\ldots , x_{e-1} \ox 1)$ is a basis of $F
    \ox_{R\prd} Q$. Now, let $(v)$ be a basis of $\ol{E} \ox_{R\prd}
    Q \simeq Q$ and let $\beta_{e}, \ldots, \beta_{n}$ be elements
    in $Q$ such that, for all $i \geq e$,
    $$ x_{i} \ox 1 = \sum_{j=1}^{e-1}\gamma_{ij}(x_{j}\ox 1)+\beta_{i} v, $$
    for some $\gamma_{ij} \in Q$.
    Let $\seq{e}{n}$ be the canonical basis of $Q^n$.
    Let 
    $\ol{h} = h \ox \Id$ and suppose that $\var = (\alp_{ij})_{i,j}$. Since 
    $\ol{h}\ol{\var}=0$ then, for all $1\leq k\leq m$,
    \begin{align*}
    0 & = \sum_{i=1}^{e-1} \alpha_{ik}(x_{i}\ox 1) + \sum_{i=e}^{n} \alpha_{ik}(x_{i}\ox 1) \\
    & = \sum_{i=1}^{e-1} \alpha_{ik}(x_{i}\ox 1) + \sum_{i=e}^{n} \alpha_{ik}\left(\sum_{j=1}^{e-1}\gamma_{ij}(x_{j}\ox 1)+\beta_{i} v\right) \\
    & = \sum_{i=1}^{e-1} \alpha_{ik}(x_{i}\ox 1) + \sum_{i=1}^{e-1}\left(\sum_{j=e}^{n}\alpha_{jk}\gamma_{ji}\right)(x_{i}\ox 1)+\left(\sum_{i=e}^{n}\alpha_{ik}\beta_{i}\right)v \\
    & = \sum_{i=1}^{e-1} \left(\alpha_{ik}+\sum_{j=e}^{n}\alpha_{jk}\gamma_{ji}\right)(x_{i}\ox 1)+\left(\sum_{i=e}^{n}\alpha_{ik}\beta_{i}\right)v.
    \end{align*}
   Therefore $\sum_{i=e}^{n}\alpha_{ik}\beta_{i}=0$ and each row $i$ of $\var$, with $1 \leq i \leq e-1$, is a $Q$-linear combination of the rows $e,\ldots,n$. That is
$$ \var = \left[ \begin{array}{c|c} I_{e-1} & - \Gamma^{T} \\ \hline 0 & I_{n-e+1} \end{array} \right] \left[ \begin{array}{c}  0 \\ \hline \Psi
    \end{array} \right], $$
where
$\Gamma =(\gamma_{k\ell})$ is a $Q$-matrix of type $(e-1) \x (n-e+1)$
and $\Psi= \var[e,\ldots,n | 1,\ldots,m]$.
Moreover, $I_{n-e}(\Psi) \cdot Q = I_{n-e}(\var) \cdot Q =Q$, and so $\grd I_{n-e}(\Psi) \geq 1$, as claimed.
\end{proof}

Let $\psi$ be an $(n-e+1) \times (n-e)$ submatrix of
$\varphi$ satisfying $\grd I_{n-e}(\psi) \geq 1$ (exists by Lemma \ref{phi}).
Since $\rank \im \psi =
n-e$ (again by \cite[Proposition~1.4.11]{bh}), $\rank \ker \psi = 0$.
But $\ker \psi$ is torsionfree, hence $\ker \psi = 0$ and $\psi$
represents an injective linear map
$$ 0 \ra {R\prd}^{n-e} \overset{\psi}{\ra} {R\prd}^{n-e+1}$$
with respect to the canonical bases.
Moreover, using Eq. (\ref{var}),
we get for any $j=1,\ldots,m$
$$ \sum_{i=e}^n x_{i}
\alp_{ij} = - \sum_{i =1}^{e-1} x_{i}
\alp_{ij} \in F.$$
Therefore, for any $j$,
$\sum_{i =e}^n \ol{x_{i}} \alp_{ij} = 0,$
where $\ol{x_{i}}= x_{i} + F$. Hence
\begin{equation} \label{var2}
[\ol{x_{e}}\, \cdots\, \ol{x_{n}}] \psi = 0.
\end{equation}

On the other hand, let $\sek{u}{n-e} \in {R\prd}^{n-e+1}$ be the
columns of $\psi$ and consider the map
$\apl{\theta}{\bwd^1 {R\prd}^{n-e+1}}{\bwd^{n-e+1} {R\prd}^{n-e+1}}
\simeq R\prd$
given by $x \mapsto x \wdg u_{1} \wdg \cdots \wdg u_{n-e}$.
We have, by \refl{ext3},
$$ \im \tet = I_{n-e}(\psi). $$
Moreover,
$$ x \in \im \psi \iff  x \in \langle \sek{u}{n-e} \rangle \Lra x
\wdg u_{1} \wdg \cdots \wdg u_{n-e} = 0 \iff x \in \ker \tet, $$
proving that $\im \psi \sube \ker \tet$.
Therefore, there exists a complex of finite free $R\prd$-modules
$$ 0 \ra {R\prd}^{n-e} \overset{\psi}{\ra} {R\prd}^{n-e+1}
\overset{\tet}{\ra} R\prd \ra 0 $$
with $\im \tet = I_{n-e}(\psi) \sub R\prd$.

Now, suppose in addition that $E$ satisfies $\grd F_{e}(E) \geq 2$. In this case, as we already observed, $\ol{E}=E''/F$ is a \fg{} \tf $R''$-module having rank $1$.
Hence, $\ol{E} \simeq I_{U}(E)=I$ a \gbiu.
Since $I_{n-e}(\psi) = \langle | \psi_{e}|,\ldots, | \psi_{n}|
\rangle$, where $\psi_{i}$ is obtained from $\psi$ by deleting the
$(i-e+1)$-th row $(i=e,\ldots,n)$, and $I= \langle
\ol{x_{e}},\ldots,\ol{x_{n}} \rangle$ the mapping $\ol{x_{i}}
\mapsto (-1)^{i-e}|\psi_{i}|$ $(i=e,\ldots,n)$ extends to an
$R\prd$-epimorphism
$\apl{\rho}{I}{I_{n-e}(\psi)}$. In fact, if $\sum_{i=e}^n
\beta_{i} \ol{x_{i}} = 0$ and writing $\psi\pr =
\begin{bmatrix}
    \beta_{e} &  \\
    \vdots &  \psi \\
    \beta_{n}  &
\end{bmatrix}$
then
$[\ol{x_{e}} \, \cdots \, \ol{x_{n}}] \psi\pr =0$,
and we deduce that
$\ol{x_{i}} | \psi\pr | = 0$, $(i=e,\ldots,n)$.
Since $I= \langle
\ol{x_{e}},\ldots,\ol{x_{n}} \rangle$, $I |\psi\pr|=0$ and so
$|\psi\pr| \in \ann_{R\prd}(I)$. But $\rank I = 1 >0$, hence
$\ann_{R\prd}(I)=0$.
Therefore
$$ 0 = |\psi\pr| = \sum_{i=e}^n \beta_{i} (-1)^{i-e} |\psi_{i}|, $$
(by Laplace's theorem) proving that $\ol{x_{i}}
\mapsto (-1)^{i-e}|\psi_{i}|$ is well defined.
Hence there exists an $R\prd$-epimorphism $I \overset{\rho}{\ra}
I_{n-e}(\psi)$.
Moreover, since $I$, $I_{n-e}(\psi)$
are $R\prd$-ideals of rank $1$ (and every ideal is a torsionfree
$R\prd$-module),
$\ker \rho = 0$ and we have
\begin{equation}\label{ii}
I \simeq I_{n-e}(\psi).
\end{equation}

Now, suppose that $\grd I_{n-e}(\psi) \geq 2$.
By Hilbert-Burch Theorem (\cite[Theorem 1.4.17]{bh}),
$I_{n-e}(\psi)$ has the free resolution
$$ 0 \ra {R\prd}^{n-e} \overset{\psi}{\ra} {R\prd}^{n-e+1}
\overset{\tet}{\ra} I_{n-e}(\psi) \ra 0.$$
Moreover, $I_{n-e}(\psi)\simeq R\prd$ or $I_{n-e}(\psi)$ is perfect of
grade $2$. On the other hand,
by Eq. (\ref{ii}),
$I_{n-e}(\psi) \simeq I_{U}(E)=I$ a \gbiu{} ($I_{U}(E)\simeq E\prd/F$).
If $E$ is not free,
the Bourbaki ideal $I \simeq \langle \ol{x_{e}},\ldots, \ol{x_{n}}
\rangle$ is also not free (by \reft{free2}). As a consequence, $I$ has a free resolution
$$ 0 \ra {R\prd}^{n-e} \overset{\psi}{\ra} {R\prd}^{n-e+1}
\ra I \ra 0, $$
and we get that $I=I_{n-e}(\psi)$ is perfect of grade $2$.
Hence $\psi$ defines a presentation of $I$,
with respect to $\ol{x_{e}}, \ldots, \ol{x_{n}}$.
\smallbreak
We have proved the following result.

\begin{theorem}\label{fi}
Let $R$ be a \nol, $E$ be a \fgr{} having rank $e \geq 2$ and $U$ a
reduction of $E$.
Let $R\prd$, $E\prd$ and $F= {\+}_{i=1}^{e-1}R\prd
x_{i}$ as before.
Let $\sek{x}{n}$ be a generating set of $E\prd$ containing the
basis $\sek{x}{e-1}$ of $F$.
Let $\var$ be an $n \times m$ matrix
presenting $E\prd$
with respect to the generators $\sek{x}{n}$ and let $\psi$ be an $(n-e+1)\x(n-e)$ submatrix of $\var$ satisfying $\grd I_{n-e}(\psi) \geq 1$.
Then there exists a complex of the form
$$ 0 \ra {R\prd}^{n-e} \overset{\psi}{\ra} {R\prd}^{n-e+1}
\ra I_{n-e}(\psi) \ra 0. $$
Suppose in addition that $E$ satisfies $\grd F_{e}(E) \geq 2$.
Then $E\prd/F$ is isomorphic to an $R\prd$-ideal $I=I_{U}(E)$ (a \gbiu) and
$$I \simeq I_{n-e}(\psi).$$

Moreover, if $E$ is not free
and $\grd I_{n-e}(\psi) \geq 2$ then $I$ is perfect of grade $2$ with a finite free resolution
$$ 0 \ra {R\prd}^{n-e} \overset{\psi}{\ra} {R\prd}^{n-e+1}
\ra I \ra 0$$
and $\psi$ defines a presentation of $I$ with respect to the generators $\ol{x_{e}}, \ldots, \ol{x_{n}}$,
with $I=I_{n-e}(\psi)$.
\end{theorem}

\begin{corollary}\label{fi1}
    Let $R$ be a \nol, $E$ an ideal module having rank $e \geq 2$
    and $U$ a reduction of $E$.
    Then any \gbiu{} is isomorphic to a Fitting ideal.
\end{corollary}

\begin{proof}
    Let $I=I_{U}(E) \sub R\prd$ be a \gbiu. Since $\grd F_{e}(E)
    \geq 2$ then $I \simeq I_{n-e}(\psi)$, where $\psi$ is a submatrix
    of a matrix $\var$ presenting $E\prd$ (by the theorem above), as
    required.
\end{proof}

For certain modules \gbi s can be chosen to have grade $\geq 2$, and we called them {\it \ggbi s} (see \cite{cz3}). In fact, it is proved in \cite[Proposition~3.2]{suv} that a \fg{} \rmo{} $E$ having rank $e>0$ and satisfying $\grd F_{e}(E) \geq 2$ has a \ggbi{} \sse $E$ is orientable.
In particular, ideal modules also have \ggbi s.
As already observed, $V(F_{e}(E)) = \supp G/E = \supp R/F_{e}(E)$ in the case where $E \sube G \simeq R^{e}$ is an ideal module. In particular, $V(F_{1}(I)) = \supp R/I = V(I)$ for any $R$-ideal $I$ with $\grd I \geq 2$.
Moreover, if $I \simeq E\prd/F$ is a \ggbie{} then, clearly,
\begin{equation}\label{varf1fe}
V(F_{e}(E\prd)) \sube V(F_{1}(I)) = V(I).
\end{equation}

The following result characterizes ideal modules with projective dimension equal to one via \gbi s. Using the notations of Theorem \ref{fi}, this shows in particular that if $E$ is not free and $\grd I_{n-e}(\psi) \geq 2$ then $\pjd E=1$.

\begin{proposition}\label{fi2}
Let $R$ be a \nol, $E$ an ideal module having rank $e \geq 2$ and $U$ a reduction of $E$. Then the following are equivalent:
\begin{enumerate}
\item $\pjd E=1$.
\item Any \ggbie{} (\wrt $U$) is perfect of grade $2$.
\item There exists a \gbie{} (\wrt $U$) which is  perfect of grade $2$.
\end{enumerate}
\end{proposition}

\begin{proof}
Suppose that $\ol{E} \simeq I_{U}(E)=I$ is an $R''$-ideal.
By \reft{free2}
$$ \pjd R''/I = \pjd \ol{E} +1 = \pjd E+1. $$
Since $\grd I \leq \pjd R''/I$ the equivalences then follow.
\end{proof}

\begin{remark}\label{ht}
Let $I$ be a perfect ideal of grade $2$. Then, $\Ht I \leq 2$ by \cite[Theo\-rem~1.1.18]{wv1}, and because $\grd I \leq \Ht I$ we have that $\Ht I = 2$. This can be extended to modules of projective dimension $1$ in the following way: Let $E$ be a \fgrmo{} having rank $e$. Suppose that $\pjd E=1$. Then
$\Ht F_{e}(E)\leq e+1.$
\end{remark}

\begin{proof}
Let $0 \ra R^m \overset{\var}{\ra} R^n \ra E \ra 0$ be a (minimal) free resolution of $E$. Since $\rank E=e$ then $m=n-e$. Moreover,
$I_{n-e}(\varphi)\neq 0$ and $I_{n-e+1}(\varphi)= 0$. Thus, by \cite[Theorem~1.1.18]{wv1},
$$ \Ht F_{e}(E)=\Ht I_{n-e}(\var) \leq
e+1$$
with $m=n-e=t$.
\end{proof}

In \cite{cz2} we defined the analytic deviation, for an ideal module $E \subsetneq G \simeq R^{e}$ with positive rank $e>0$, as $\ad(E)=\lE-e+1-\Ht F_{e}(E)$. We say that $E$ is \eq{} if $\ad(E)=0$. Moreover, we say that $E$ is a module of the principal class if $\mu(E)=\Ht F_{e}(E)+e-1$ and that $E$ is a \ci{} if $\mu(E)=\grd F_{e}(E)+e-1$.
It is clear that these notions agree with the correspondent ones for ideals. Moreover, since $\mu(E) \geq \lE$ (by Eq. (\ref{mule})), $\lE \geq \Ht F_{e}(E) +e -1$ (by \cite[Proposition~3.12]{cz2}), and $\grd F_{e}(E)= \grd G/E \geq 2$, then
\begin{equation} \label{leqs}
\mu(E) \geq \lE \geq \Ht F_{e}(E) +e -1 \geq \grade F_{e}(E) +e -1 \geq e+1.
\end{equation}

Suppose that $I$ is a \ggbi{} of $E$. Since $I \simeq E''/F$ we always have $\mu(E) \leq \mu(I) +e-1$. On the other hand, from Eq. (\ref{varf1fe}) we have that $\Ht I \leq \Ht F_e(E'') = \Ht F_e(E)$. Hence, if $I$ is an ideal of the principal class [\ci{} or \eq,
respectively] then $E$ satisfies the same property.
However we are interested in the other situation, i.e. to know when a \gbie{} satisfies the same property as the module $E$. If $\dim R = 2$ this is always true, since then $E$ is of the principal class if and only if $\mu (E) = e + 1$ and so by Theorem \ref{free2} we have $\mu (I) = 2 = \Ht (I)$ (note that $R$ must be Cohen-Macaulay and so to be of the principal class is the same as to be a \ci{}). And similarly for the equimultiple property by using the equality (\ref{Ebarra1}).
In fact, in the general case case where $\mu(E)$ or $\lE$ reach the minimum values we have the following result.

\begin{proposition}\label{e+1}
Let $\Rmk$ be a \nol{} and let $E \subn G \simeq R^e$ be an ideal module having rank $e >0$.
\begin{enumerate}
\item If $\mu(E)=e+1$ then $E$ is a module of the principal class.
\item If $\lE=e+1$ and $\frk$ is infinite then $E$ is an \eq{} module.
\end{enumerate}
In both cases, if $I$ is a \ggbie \, then $I$ satisfies the same properties as $E$ and we have
\begin{equation}\label{ig}
 \grd F_{e}(E)=\Ht F_{e}(E)=2=\Ht I = \grd I.
 \end{equation}
\end{proposition}

\begin{proof}
(a) If $\mu(E)=e+1$ then $\lE=e+1$ (by Eq. (\ref{leqs})). Hence
$$ 2 \leq \grd F_{e}(E) \leq \Ht F_{e}(E) \leq \mu(E) -e+1 =2 $$
and so
$$\mu(E)=\lE=\Ht F_{e}(E) +e-1=\grd F_{e}(E) +e-1,$$
proving that $E$ is of the principal class.

(b) Let $U$ be a minimal reduction of $E$. Then $\mu(U)=\lE=e+1$. By (a), $U$ is of the principal class. Then $E$ is an \eq{} module (by \cite[Proposition~4.3]{cz2}) and we have
$$ \grd F_{e}(E) = \Ht F_{e}(E) = 2.$$

\smallbreak

Now, let $I$ be a \ggbie. Then, using Eq. (\ref{Ebarra1}) and (\ref{Ebarra4}), we get for (a)
$$ \mu(I) = 2 = \lI \; , \; 2 \leq \grd I \leq \Ht I \leq \lI=2$$
and so
$$ \mu(I)=\lI=\Ht I = \grd I = 2.$$
Therefore, $I$ is of the principal class. In the same way, supposing that $\lE=e+1$, then
$\lI=2$, and so
$$ \lI=\Ht I = \grd I = 2.$$
Therefore, $I$ is an \eq{} ideal.
In both cases, the equalities (\ref{ig}) hold.
\end{proof}

\section{Divisors of a module - part 2}\label{div2}

So far we have proved in Section \ref{div} that if $E$ is a finitely generated $R$-module with rank $e$, and
that if
\begin{equation*} \label{pE2}
R^m \overset{\varphi}{\ra} R^n \overset{\phi}{\ra} E \ra 0
\end{equation*}
is a finite presentation of $E$, then
$$\deter _0(E) \simeq [[E]] \simeq I_{n-e} (\rho) = F_e(E_1) \subseteq F_e(E),$$
where $\rho$ is an $n \times (n-e)$ submatrix of $\varphi$ with a non-zero $(n-e) \times (n-e)$ minor, and $E_1$ is an $R$-module of projective dimension $1$ with rank $e$ with a finite presentation given by $\rho$. Moreover, if $E\subseteq G \simeq R^e$ then
$$F_0(G/E) = \deter _0(E).$$

Now, we want to include in this context the Bourbaki ideal. Since we cannot do this directly over the ring $R$, we have to extend it previously to an adequate Nagata extension $R^{''}$. And this inclusion will be possible by means of the norm ideal $[[E]]_R$ and the fact proven in the previous section that any generic Bourbaki ideal is always isomorphic to a Fitting ideal.
\smallbreak

We begin by observing that the norm ideal behaves well under the extension of scalars by flat homomorphisms.

\begin{proposition}\label{esc}
Let $\apl{h}{R}{S}$ be a flat homomorphism of rings. Then
$$ [[E \oxr S]]_{S} \simeq [[E]]_{R} \ox_{R}S. $$
\end{proposition}

\begin{proof}
It is known that if $E$ has rank and $\apl{h}{R}{S}$ is flat, then $E \oxr S$ has rank and $\rank _R E = \rank _S E \oxr S$. Then, the statement follows by Eq. (\ref{e3}) and Properties \ref{dfprop}.
\end{proof}

In particular, supposing that $R\prd$ is a Nagata extension of $R$ and $E\prd= E\oxr R\prd$, as in the previous sections, we get
\begin{equation}\label{norma}
[[E'']]_{R''}= [[E \oxr R'']]_{R''} \simeq [[E]]_{R} \ox_{R}R'' = \deter_{0}(E) \cdot  R''.
\end{equation}

\smallbreak

In the notation of section \ref{qm}, suppose that $U=\sum_{i=1}^{k}Ra_{i}$ is a reduction of $E$ and $F=\sum_{j=1}^{e-1}R''x_{j}$ with $x_{j} = \sum_{i=1}^k z_{ij} a_{i}$.
Let $\{\sek{x}{n}\}$ be a generating set of $E\prd$ containing the basis $\seq{x}{e-1}$ of $F$.
Let
$$ {R\prd}^m \overset{\varphi}{\ra} {R\prd}^n \ra E\prd \ra 0 $$
be a finite free presentation of $E\prd$ \wrt the generators
$\sek{x}{n}$, and write also by $\var$ the respective relation
matrix. As we observed, in section \ref{gbifi}, there exists an $(n-e+1)\x(n-e)$ submatrix $\psi$ of $\varphi$ satisfying $\grade I_{n-e}(\psi) \geq 1$. Moreover, if $\grade F_{e}(E) \geq 2$ then $I_{n-e}(\psi) \simeq E''/F \simeq I$ (Theorem \ref{fi}) with $I$ a \gbie.
Now let $\rho$ be the $n\x(n-e)$ submatrix of $\varphi$ containing $\psi$. The matrix $\rho$ defines a finite free presentation
$$ 0 \ra {R''}^{n-e} \overset{\rho}{\ra} {R''}^{n} \ra E_{1}'' \ra 0 $$
with $E_{1}''$ an $R''$-module satisfying $\rank(E_{1}'')=e$, $E_{1}''/\tau_{R''}(E_{1}'')\simeq E''$ (as observed in section \ref{div}).
Therefore
\begin{equation}\label{rhovarphi}
I_{n-e}(\psi) \subseteq I_{n-e}(\rho) \subseteq I_{n-e}(\varphi).
\end{equation}
Moreover, by Eq. (\ref{norma})
\begin{equation}\label{rhovarphi1}
\deter_{0}(E)\cdot R'' \simeq [[E'']]_{R''} \simeq F_{e}(E_{1}'')
\subseteq F_{e}(E'') \simeq F_{e}(E) \oxr R'' = F_{e}(E)\cdot R''.
\end{equation}
Further, if $\grade F_{e}(E) \geq 2$ then
\begin{equation}\label{rhovarphi2}
I \simeq I_{n-e}(\psi) \subseteq I_{n-e}(\rho) = F_{e}(E_{1}'') \simeq [[E'']]_{R''}.
\end{equation}

In the case where $E$ has projective dimension equal to one, we may assert in addition that $F_{e}(E'')$ is a representative of $[[E'']]_{R''}$.

\begin{proposition}\label{norma2}
Let $R$ be a \nol{} and  $E$ a \fgrmo{} having rank $e \geq 2$. If $\grade F_{e}(E) \geq 2$ then
$$ I \simeq I_{n-e} (\psi) \subseteq I_{n-e}(\rho) \subseteq F_{e}(E'')\,,$$
with $I_{n-e}(\rho) \simeq [[E'']]_{R''} \simeq \deter_{0}(E)\cdot R''$, for any \gbi{} $I$ of $E$.
If moreover $\pjd E=1$, then
$$ I \simeq I_{n-e} (\psi) \subseteq I_{n-e}(\rho) = F_{e}(E'')$$
for any \gbi{} $I$ of $E$.
\end{proposition}

\begin{proof}
Follows by Eq. (\ref{rhovarphi}) and (\ref{rhovarphi2}).
\end{proof}

Next suppose that $\Rm$ is a $2$-dimensional regular local ring and $E$ is a \fgtfrmo with rank $e$. Recall that $E$ is said to be {\it contracted} if $E=ES \cap R^{e}$, where $S=R\left[\frac{\fm}{a}\right]$, $a$ is a minimal generator of $\fm$ and $ES$ is the $S$-submodule of $S^{e}$ generated by $S$. Contracted modules were defined by V. Kodiyalam in \cite{ko} as an extension to modules of the notion of contracted ideal. Integrally closed modules over a $2$-dimensional regular local ring $R$ are contracted. Given a contracted module $E$, it may be seen that $G = E^{**}$ is free and that $G/E$ is of finite length, so $\grade G/E =2$ and $E$ is an ideal module, see \cite[Proposition 2.1]{ko}.

It then happens that contracted modules have a special generic Bourbaki ideal: It is proved in \cite[Corollary 3.6]{hu} that if $E$ is contracted then $F_{e}(E'')$ is a \ggbi. So as a consequence of all the above relations we have the following:

\begin{corollary} \label{norma3}
Let $\Rm$ be a $2$-dimensional regular local ring and $E$ a \fgtfrmo with rank $e$ which is not free. If $E$ is contracted then there exists a \ggbi{} $I$ of $E$ which is a representative of $[[E'']]_{R''}$ and in this case $$ I=F_{e}(E'') \simeq [[E'']]_{R''} \simeq \deter_{0}(E)\cdot R''. $$
\end{corollary}

\begin{proof}
Since $E$ is an ideal module, $\grade F_e(E) = 2$. Also, $\pjd E=1$. The result then follows by Proposition \ref{norma2}.
\end{proof}

We finish by pointing out that, in the above conditions, the blow up at any \gbi{} has an universal flattening property under birational morphisms, as we described in Theorem \ref{blow}. For instance, if $E$ is an integrally closed module over a $2$-dimensional regular local ring $R$, we have that the Rees algebra of $E$ is Cohen-Macaulay. So after a suitable Nagata extension $R''$ of $R$, this Rees algebra is a deformation of the Rees algebra of a \ggbi{} of $E$, whose blow up has such universal flattening property with respect to $E'' = E \otimes_R R''$. It would be interesting to know other instances where the Rees algebra of a module satisfies a similar universal property.


\end{document}